\documentclass{amsart}

\usepackage{amssymb,latexsym,amsmath,extarrows,amsthm, tikz}
\usepackage{graphicx}
\usepackage{xcolor}
\usepackage{mathabx}

\newtheorem*{theorem*}{Theorem}

\newtheorem{theorem}{Theorem}[section]
\newtheorem{lemma}[theorem]{Lemma}
\newtheorem{proposition}[theorem]{Proposition}
\newtheorem{remark}[theorem]{Remark}

\newtheorem{definition}[theorem]{Definition}
\newtheorem{corollary}[theorem]{Corollary}

\newcommand{\CC}{\mathbb{C}}
\newcommand{\RR}{\mathbb{R}}
\newcommand{\TT}{\mathbb{T}}
\newcommand{\ZZ}{\mathbb{Z}}
\newcommand{\QQQ}{\mathcal{Q}}
\newcommand{\WW}{\mathbb{W}}
\newcommand{\mug}{\mu_{1, good}}
\newcommand{\supp}{\textrm{supp}}
\newcommand{\RD}{\textrm{RapDec}}
\newcommand{\Bad}{\textrm{Bad}}
\newcommand{\HH}{\mathcal{H}}

\begin{document}
\title[2D Falconer]{On Falconer's distance set problem in the plane}
\author{Larry Guth, Alex Iosevich, Yumeng Ou, and Hong Wang}
%\centerline{\today}
\begin{abstract} If $E \subset \RR^2$ is a compact set of Hausdorff dimension greater than $5/4$, we prove that there is a point $x \in E$ so that the set of distances $\{ |x-y| \}_{y \in E}$ has positive Lebesgue measure.
\end{abstract}

\maketitle

\section{introduction}

\vskip.125in 

For a set $E \subset \RR^d$, define the distance set 

$$\Delta(E)=\{|p-p'|: p,p' \in E\}.$$

\noindent Falconer's distance problem (\cite{Falc86}) is about the connection between the Hausdorff dimension of a set $E$ and the size of $\Delta(E)$.   Given a compact set $E$ in ${\Bbb R}^d$, $d \ge 2$, Falconer asked how large the Hausdorff dimension of $E$ needs to be to ensure that the Lebesgue measure of $\Delta(E)$ is positive. He proved that if $dim_{{\mathcal H}}(E)>\frac{d+1}{2}$, then ${\mathcal L}(\Delta(E))>0$. Using an example based on the integer lattice, he showed for every $s \leq \frac{d}{2}$ there exist sets of Hausdorff dimension $s$ for which ${\mathcal L}(\Delta(E))=0$. This led him to conjecture that if $dim_{{\mathcal H}}(E)>\frac{d}{2}$, then the Lebesgue measure of the distance set is positive. This is known as the Falconer Distance Conjecture. 

In \cite{W99}, Wolff proved that if $E \subset \RR^2$ is a compact set with Hausdorff dimension greater than $4/3$, then $\Delta(E)$ has positive Lebesgue measure.  In this paper, we improve the bound.

\begin{theorem} \label{main1} If $E \subset \RR^2$ is a compact set with Hausdorff dimension greater than $5/4$, then $\Delta(E)$ has positive Lebesgue measure.
\end{theorem}

In higher dimensions, Erdo\u{g}an proved in \cite{Erd05}  that if $dim_{{\mathcal H}}(E)>\frac{d}{2}+\frac{1}{3}$, then ${\mathcal L}(\Delta(E))>0$.  Recently, these estimates were improved for all $d \ge 3$ by Du, Guth, Ou, Wang, Wilson, and Zhang \cite{DGOWWZ18}.  In dimension 3, they showed that the Falconer conjecture holds when $dim_{\HH}(E) > 9/5$.  The estimates for $d \ge 4$ were further improved by Du and Zhang in \cite{DZ18}.  For large $d$, they prove that Falconer's conjecture holds when $dim_{\HH}(E) > \frac{d}{2} + \frac{1}{4} + o(1)$.  These works brought into play the decoupling theorem of Bourgain and Demeter \cite{BD15}.  This approach will also play a key role in our proof.

Returning to the planar case, there have been a number of important recent results.  Orponen \cite{O17} proved that if $E$ is a compact Ahlfors-David regular set of dimension $s \ge 1$, then $\Delta(E)$ has packing dimension 1.  Note that packing dimension 1 is only slightly weaker than positive measure.  This result was striking because in previous work on the problem, there was no evidence that the Ahlfors-David case would be any easier than the general case.  This approach was further developed by Keleti and Shmerkin \cite{KS18}.  They proved very strong estimates for sets that are even roughly like Ahlfors-David regular sets.  They also proved results about the Hausdorff dimension of $\Delta (E)$.  For instance, if $E$ is a compact set with Hausdorff dimension strictly greater than $1$, then they proved that the Hausdorff dimension of $\Delta(E)$ is at least $.685...$.   Bourgain (\cite{B03}) had proven that if $E$ has Hausdorff dimension at least 1, then $\Delta(E)$ has Hausdorff dimension at least $1/2 + \delta$ for some $\delta > 0$.  The value of $\delta$ could be made explicit but it would be very small, and so the $.685...$ is quite striking. We will use one of the key ideas of \cite{O17} and \cite{KS18} in the proof of Theorem \ref{main1}. 
  
There is a variant of the Falconer Distance Problem involving pinned distance sets.  For any point $x$, the pinned distance set $\Delta_x(E)$ is defined by

$$ \Delta_x(E)=\{|x-y|: y \in E\} .$$ 

\noindent Peres and Schlag (\cite{PS00}) proved that if $E \subset {\Bbb R}^d$, $d \ge 2$ and $dim_{{\mathcal H}}(E)>\frac{d+1}{2}$, then ${\mathcal L}(\Delta_x(E))>0$ for every $x \in E$ except for a set of small Hausdorff dimension. Improvements on the size of the exceptional set were obtained by the second listed author and Liu in \cite{ILiu17}.  

Recently, in \cite{Liu18}, Liu showed that if $dim_{{\mathcal H}}(E)>\frac{d}{2}+\frac{1}{3}$, then 
${\mathcal L}(\Delta_x(E)) > 0$ for every $x \in E$ except those in a set of small Hausdorff dimension.  Using Liu's method, we are also able to bound the size of pinned distance sets.  

\begin{theorem}\label{main}
If $E\subset\mathbb{R}^2$ is a compact set with Hausdorff dimension larger than $\frac54$, then there is a point $x\in E$ such that its pinned distance set 
$\Delta_x(E)$ has positive Lebesgue measure.
\end{theorem}

\subsection{Other norms}

The Falconer problem has also been studied for other norms.  Suppose that $K$ is a symmetric convex body in $\RR^d$ and $\| \cdot \|_K$ is the norm with unit ball $K$.  We let $\Delta_{K}(E)$ be the set of distances $ \| x-y \|_{K}$ with $x,y \in E$ and we let $\Delta_{K,x}(E)$ be the set of distances $\| x - y \|_K$ with $y \in E$.  If $K$ is the cube $[-1,1]^d$, then $\| \cdot \|_K$ is the $l^\infty$ norm, and it is not difficult to construct a compact set $E \subset \RR^d$ with Hausdorff dimension $d$ so that $\Delta_{K}(E)$ has measure zero.  But there are non-trivial results if $K$ is curved.  We focus on the case that $\partial K$ is $C^\infty$ smooth and has positive Gaussian curvature.  It is plausible that Falconer's conjecture remains true for all such norms, and most previous results on the problem extend to this setting.  For instance, Erdo\u{g}an's bound extends to this class of norms -- cf. Remark 1.6 in \cite{Erd05}.  Our method also extends to this class of norms.

\begin{theorem} \label{main2} Let $K$ be a symmetric convex body in ${\Bbb R}^2$ whose boundary $\partial K$ is $C^\infty$ smooth and has strictly positive curvature.  Let $E\subset\mathbb{R}^2$ be a compact set whose Hausdorff dimension is larger than $\frac54$. Then, there exists a point $x\in E$ so that the pinned distance set $$\Delta_{K,x}(E):=\{{||x-y||}_K:\, y\in E\}$$ has positive Lebesgue measure.
\end{theorem}

%Keleti and Shmerkin (\cite{KS18}, Theorem 1.1)proved that if the Hausdorff dimension of a compact set $E \subset {\Bbb R}^2$ is equal to $s > 1$, then there exists $x \in E$ so that the Hausdorff dimension of $\Delta_x(E)$ is at least $\frac{2}{3}s$. The proofs of Theorem \ref{main1} and Theorem \ref{main2} below can be adapted to yield the following. 

\begin{remark} \label{ksanalogrmk} One can adapt the proof of Theorem \ref{main1} and Theorem \ref{main2} to yield the following result. Suppose that the Hausdorff dimension of a compact set $E \subset {\Bbb R}^2$ is equal to $s>1$ and $K$ is as in Theorem \ref{main2}. Then there exists $x \in E$ such that the upper Minkowski dimension of $\Delta_{x, K}(E)$ is $\ge \frac{4s}{3}-\frac{2}{3}$.  Keleti and Shmerkin (\cite{KS18}) obtained the lower bound $\frac{1}{4}(1+s+\sqrt{3s(2-s)})$ in the case of the Euclidean metric. Their estimate is better than ours near $s=1$, but ours is preferable as $s$ nears $\frac{5}{4}$. The sketch of this argument is given in Appendix where we also discuss the complications of replacing the upper Minkowski dimension by the Hausdorff dimension in the claim above. \end{remark} 

\vskip.125in 

Falconer's distance problem can be thought of as a continuous analogue of a combinatorial problem raised by Erd\H os in \cite{Erd45}: given a set $P$ of $N$ points in $\RR^d$, what is the smallest possible cardinality of $\Delta(P)$.  A grid is the best known example in all dimensions.  In two dimensions, Guth and Katz \cite{GK15} proved a lower bound for $|\Delta(P)|$ which nearly matches the grid example (up to a factor of $\log^{1/2} N$).  In higher dimensions, there is a larger gap, and the best known result is due to Solymosi and Vu \cite{SV08}.  The Erd\H os distinct distance problem also makes sense for general norms and much less is known about it.  In the planar case, if $K$ is smooth and has strictly positive curvature, the best known bound says that if $|P| = N$, then $| \Delta_K(P) | \gtrsim N^{3/4}$, with stronger estimates established by Garibaldi in special cases (\cite{Gar04}).  There is a conversion mechanism to go from Falconer-type results to Erd\H os-type results that was developed by the second author together with Hoffman (\cite{HI05}), Laba (\cite{IL05}), and Rudnev and Uriarte-Tuero (\cite{IRU14}).  It gives estimates for point sets that are fairly spread out.  Applying the conversion mechanism to Theorem \ref{main2} we get the following corollary:

\begin{corollary} \label{falconertoerdosthm}  Let $K$ be a symmetric convex body in ${\Bbb R}^2$ whose boundary $\partial K$ is $C^\infty$ smooth and has strictly positive curvature.  Let $P$ be a set of $N$ points in ${[0,1]}^2$ so that the distance between any two points is $\gtrsim N^{-1/2}$.  
	Then there exists $x \in P$ such that 
	\begin{equation} \label{erdospinnedours} | \Delta_{K,x}(P)| \gtrapprox N^{\frac{4}{5}}. \end{equation} 
\end{corollary}

\subsection{The main obstacle}

The work on the Falconer problem by Wolff \cite{W99} and Erdo\u{g}an \cite{Erd05} is based on a framework developed by Mattila (\cite{Mat85} and \cite{Mat87}) which connects the original geometric problem to estimates in Fourier analysis.   Suppose that $E$ is a compact set with positive $\alpha$-dimensional Hausdorff measure.  Then there is a probability measure $\mu$ supported on $E$ with $\mu(B(x,r)) \lesssim r^\alpha$ for every ball $B(x,r)$.  The measure $\mu$ is called a Frostman measure (cf. \cite{W03}, Proposition 8.2.).  Let $d(x,y) = |x-y|$.  Mattila considered the pushforward measure $d_*(\mu \times \mu)$.   Recall that a pushforward measure is defined by

$$ \int_\RR \psi(t) d_*(\mu \times \mu) := \int_{E \times E} \psi(|x-y|) d \mu(x) d \mu(y). $$

\noindent In particular $d_*(\mu \times \mu)$ is a probability measure supported on $\Delta(E)$.  Mattila noted that if $ \| d_*(\mu \times \mu) \|_{L^2}^2 = \int d_*(\mu \times \mu)(t)^2 dt$ is finite, then Cauchy-Schwarz forces the Lebesgue measure of $\Delta(E)$ to be positive.  Then he described an interesting way to rewrite $\| d_*(\mu \times \mu) \|_{L^2}^2$ in terms of the Fourier transform of $\mu$.  The resulting integral is connected to restriction theory, and Wolff used that connection to prove the bound in \cite{W99}, building on earlier work by Bourgain \cite{B94}.

In \cite{Liu18}, Liu used a different framework for the Falconer problem which leads to estimates on pinned distance sets.  For any $x$, define $d^x(y) = |x-y|$.  He studied the quantity

\begin{equation} \label{liuquant} \int \| d^x_* (\mu) \|_{L^2}^2 d \mu(x). \end{equation}

\noindent If this key quantity is finite, then for almost every $x \in E$, $\| d^x_* \mu \|_{L^2}$ is finite, and then a Cauchy-Schwarz argument forces the Lebesgue measure of $\Delta_x(E)$ to be positive.  Liu introduced an interesting way to rewrite this quantity in terms of the Fourier transform of $\mu$.  It can then be studied using restriction theory, leading to estimates on the pinned distance problem.  

In the planar case, there is an obstruction to pushing either one of these methods to dimensions below $4/3$.  For every $\alpha < 4/3$, there is a set $E$ of dimension $\alpha$ and a Frostman measure $\mu$ on $E$ so that $\| d_*(\mu \times \mu) \|_{L^2}$ is infinite, and also $\| d^x_* (\mu) \|_{L^2}$ is infinite for every $x \in E$.  This set is a variation on an example from \cite{KT01}.  The set $E$ looks roughly like several parallel train tracks.  In the following figure, we show an approximation of the set $E$ at a small scale $R^{-1}$.  The measure $\mu$ (approximated at scale $R^{-1}$) is just the normalized area measure on this set.

\begin{center}
\includegraphics[scale=1]{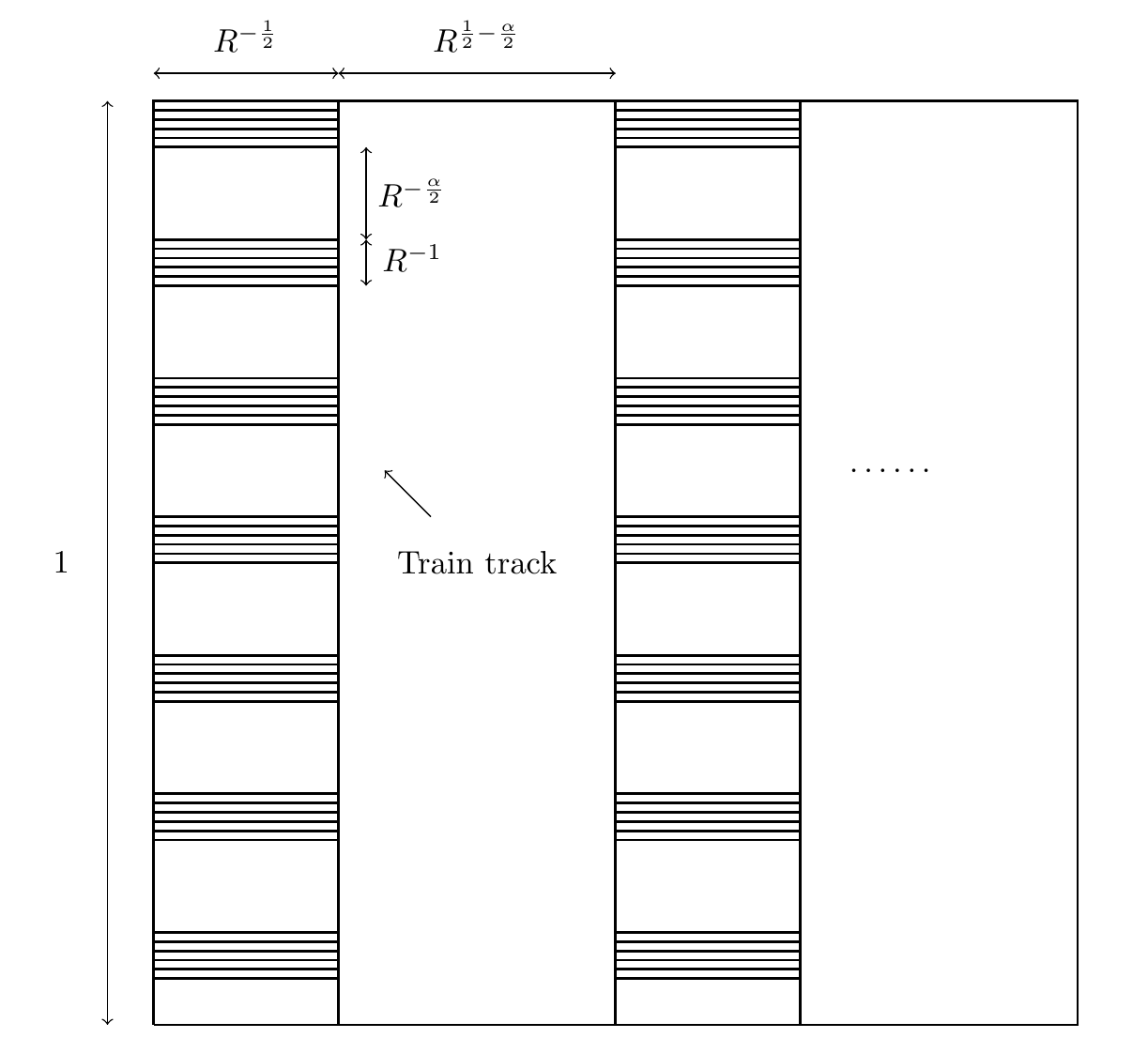}
Figure 1 
\end{center}

The set $E$ is divided among several large $R^{-1/2} \times 1$ rectangles.  Within each of these large rectangles, the set $E$ consists of evenly spaced parallel rectangles with dimensions $R^{-1/2} \times R^{-1}$.  Each of these smaller rectangles is called a slat.  The restriction of $E$ to one of the larger rectangles is called a train track.  The spacing between two consecutive slats is controlled by the dimension of $E$, and it works out to $R^{-\alpha/2}$.  If $x$ and $y$ are in the same train track, on roughly opposite sides, and if $y$ is $M$ slats from $x$, then $|x-y|$ lies in the interval 

$$I_M := [ M R^{-\alpha/2} - R^{-1}, M R^{-\alpha/2} + R^{-1}].$$

\noindent   Therefore, $ d_*(\mu \times \mu)$ assigns a lot of mass to the union of the intervals $I_M$.  This union is quite small, and even though the mass involved is significantly less than 1, it is still enough to force $\int |d_*(\mu \times \mu)|^2$ to be very large.  

There is a similar issue for $d^x_*(\mu)$.  If we fix any $x \in E$, and we let $T_0$ be the large rectangle containing $x$, then $d^x_* ( \mu|_{T_0})$ is mostly concentrated on $\cup I_M$, and this forces $\int |d^x_* \mu|^2$ to be very large.  On the other hand, if $T$ is a large rectangle which is far from $x$, then $d^x_* (\mu|_T)$ is rather evenly distributed --  in fact $d^x_*( \mu |_{T})$ is close to the pushforward of the uniform measure on $T$ with the same total mass.
 So if we graph $d^x_* (\mu)$, it has some peaks along $\cup I_M$ coming from the rectangle $T_0$ through $x$, but the bulk of $d^x_*(\mu)$ is spread rather evenly and comes from rectangles $T$ far from $x$.  In particular, the support of $d^x_*(\mu)$ indeed has positive Lebesgue measure.

This example is the main obstacle to proving the Falconer conjecture for dimensions less than 4/3.  Starting with a general Frostman measure, we separate out a part of it that resembles the train tracks in the example above.  Then we estimate the train-track part and the non-train-track part in different ways.  

For technical reasons, we consider two subsets $E_1, E_2 \subset E$ separated by distance $\sim 1$, and we let $\mu_1$ and $\mu_2$ be Frostman measures on $E_1, E_2$.  In the example above, we can imagine that $E_1$ is the bottom third and $E_2$ is the top third.  We divide $\mu_1$ into two pieces

$$ \mu_1 = \mug + \mu_{1, bad}, $$

\noindent where $\mu_{1,bad}$ is essentially the train-track-like part of $\mu_1$.   We always arrange, however, that $\int \mu_{1, bad} = 0$.

For example, if $\mu_1$ is the normalized area measure on the set $E_1$ in Figure 1 above, then $\mug$ would be (approximately) the normalized area measure on the union of the large rectangles.  The bad part, $\mu_{1, bad}$, is equal to $\mu_1 - \mug$, so it would be large on the slats and slightly negative on the parts of the large rectangles outside of the slats.  If $T_0$ is the large rectangle containing $x$, then $d^x_* (\mug|_{T_0})$ would be much more spread out than $d^x_* (\mu_1 |_{T_0})$.  On the other hand, if $T$ is far from $x$, then $d^x_* (\mug|_T)$ would be almost the same as $d^x_* (\mu_1 |_T)$.  All together, the graph of $d^x_* (\mug)$ would look like the graph of $d^x_* (\mu_1)$ with the peaks damped out.  The pushforward $d^x_*(\mug)$ would be quite evenly spread and its $L^2$ norm would be finite.  The graph of $d^x_*(\mu_{1, bad})$ would include the tall thin peaks from $d^x_*(\mu_1)$, and it would be slightly negative between the peaks.  Since the thin peaks have small mass, the $L^1$ norm of $d^x_*(\mu_{1, bad})$ would be small.

To prove Theorem \ref{main}, we will show that the features of $\mug$ and $\mu_{1, bad}$ that we just observed in the example from Figure 1 will occur for any set $E$ of dimension greater than $5/4$.  There are two main estimates.  The first estimate, in Proposition \ref{mainest1},  says that for most $x \in E_2$, $\| d^x_*(\mu_{1, bad})\|_{L^1}$ is small, and so the $L^1$ distance between  $ d^x_* (\mu_1)$ and $ d^x_*(\mug)$ is small.   The bad part, $\mu_{1, bad}$, is made from train track configurations, and that helps us analyze it.  Analyzing each individual train track is not difficult.  However, unlike in our example above, it could happen that each point lies in many different train tracks going in different directions.  To control this type of behavior, we use an estimate of Orponen from \cite{O17b} which also played a key role in Keleti and Shmerkin's work on the Falconer problem \cite{KS18}.  

The second estimate says that $d^x_* \mug$ is better behaved in $L^2$ than $d^x_* \mu_1$.  More precisely, Proposition \ref{mainest2} says that if $\alpha > 5/4$, then $\int_{E_2} \| d^x_* \mug \|_{L^2}^2$ is finite.  The proof of Proposition \ref{mainest2} is based on Liu's framework and on decoupling.  We will prove and then use a refinement of the decoupling theorem (Theorem \ref{refdec}) which is related to the refined Strichartz estimates that appear in \cite{DGL17}, \cite{DGLZ18}, and \cite{DGOWWZ18}.  This refinement of decoupling was proven independently by Xiumin Du and Ruixiang Zhang (personal communication).  It may be of independent interest.

Here is an outline of the paper.  In Section \ref{sec:outline}, we set up our framework (defining $\mug$ precisely) and outline the main estimates.  At that point, we will be able to make some further comments about the proofs of the two main propositions.  In Section \ref{sec:mainest1}, we prove Proposition \ref{mainest1}.  In Section \ref{sec:RS}, we state and prove a refinement of the decoupling theorem.  At that point, we will give some more context about this result.  Section 4 does not depend on any previous sections.  In Section \ref{sec:mainest2}, we prove Proposition \ref{mainest2} by combining Liu's framework with our decoupling tools.  This will finish the proof of Theorem \ref{main}. In Section \ref{sec:traintracks}, we present in detail the train track example that we introduced above.  In Section \ref{sec:genmetric}, we adapt our arguments to general metrics $\| \cdot \|_K$, proving Theorem \ref{main2}.  In Section \ref{sec:erdos}, we prove Corollary \ref{falconertoerdosthm}.  

\vskip10pt

{\bf Acknowledgements.} The first author is supported by a Simons Investigator grant. The second author is supported in part by the NSA Grant H98230-15-0319.  The third author is supported by NSF-DMS \#1764454.

\section{Setup and outline of the main estimates} \label{sec:outline}

Let $E \subset \RR^2$ be a compact set with positive $\alpha$-dimensional Hausdorff measure.  Without loss of generality, we can suppose that $E$ is contained in the unit disk.  Let $E_1$ and $E_2$ be subsets of $E$ with positive $\alpha$-dimensional Hausdorff measure so that the distance from $E_1$ to $E_2$ is $\gtrsim 1$.  Each subset $E_i$ admits a measure $\mu_i$ with the following two properties:

\begin{equation} \label{muprob} \mu_i \textrm{ is a probability measure supported on } E_i. \end{equation}

\begin{equation} \label{mualphadim} \mu_i (B(x, r)) \lesssim r^\alpha. \end{equation}

We will explain how to define $\mug$ by removing ``train-track like'' pieces from $\mu_1$.  Before going into the details, let us explain the features of a train track that motivate our definition of $\mug$.  Let $\mu$ be the example in Figure 1 and let $T$ be one of the $R^{-1/2} \times 1$ rectangles containing a train track of the set $E$.  One feature of $T$ is that $\mu(T)$ is large.  Because the slats of the train track are perpendicular to the direction of $T$, the Fourier transform of $\mu|_T$ is concentrated on frequencies that are in the same direction as $T$.  This is a second feature of $T$.  So to build $\mug$, we will first identify rectangles $T$ with large $\mu$ measure and call them bad rectangles.  Then for each bad rectangle $T$, we will identify the part of $\mu$ with physical support in $T$ and frequency support in the direction of $T$, and remove that part.  Here is the precise definition.  

We consider a sequence of scales $R_0$, $R_1$, $R_2$, etc.  Here $R_0$ is a large number that we will choose later and $R_j = 2^j R_0$.  Cover the annulus $R_{j-1} \le |\omega| \le R_j$ by rectangular blocks $\tau$ with dimensions approximately $R_j^{1/2} \times R_j$.  The long direction of each block $\tau$ is the radial direction.  We choose a partition of unity subordinate to this cover, so that

$$ 1 = \psi_0 + \sum_{j \ge 1, \tau} \psi_{j, \tau}. $$

Let $\delta > 0$ be a small constant.

For each $(j, \tau)$, cover the unit disk with tubes $T$ of dimensions approximately $R_j^{-1/2 + \delta} \times 1$ with the long axis parallel to the long axis of $\tau$.  Let $\TT_{j, \tau}$ be the collection of all these tubes, and let $\eta_T$ be a partition of unity subordinate to this covering, so that for each choice of $j$ and $\tau$, $ \sum_{T \in \TT_{j, \tau}} \eta_T $ is equal to 1 on the disk of radius 2.

Define an operator $M_T$ associated to a tube $T \in \TT_{j, \tau}$ by

$$ M_T f := \eta_T (\psi_{j, \tau} \hat f)^{\vee}. $$

\noindent Morally, $M_T f$ is the part of $f$ which has Fourier support in $\tau$ and physical support in $T$.  We also let $M_0 f := (\psi_0 \hat f)^{\vee}$.  We denote $\TT_j = \cup_{\tau} \TT_{j, \tau}$ and $\TT = \cup_{j \ge 1} \TT_j$.  If $f$ is a function supported on the unit disk, then
$f = M_0 f + \sum_{T \in \TT} M_T f$, up to a tiny error (see Lemma \ref{f=MTf} below for a precise statement).

We call a tube $T \in \TT_{j, \tau}$ bad if

$$ \mu_2 (T) \ge R_j^{-1/2 + 100 \delta}. $$

\noindent To get a sense of what this means, notice that the number of tubes $T \in \TT_{j, \tau}$ is $\sim R_j^{1/2 - \delta}$.  If each tube $T \in \TT_{j, \tau}$ contained the same amount of the measure $\mu_2$, then for each tube we would have $\mu_2(T) \sim R_j^{-1/2 + \delta}$.  A tube is bad, if it contains significantly more $\mu_2$ measure than this.  A tube is good if it is not bad.  Now we define $\mug$ to be the sum of contributions from all the good tubes.

$$ \mug := M_0 \mu_1 + \sum_{T \in \TT, T \textrm{ good}} M_T \mu_1. $$

We describe a couple of examples to give a sense of how $\mug$ behaves.  If $\mu_1$ is the normalized area measure on the set $E$ in Figure 1 above, then $\mug$ would be (approximately) the normalized area measure on the union of the large rectangles.  On the other hand, if we took the set $E$ in Figure 1 above and we changed it by tilting the slats at a 45 degree angle while keeping the large rectangles vertical, then $\mug$ would be essentially equal to $\mu_1$.  In general, $\mug$ may not be real-valued, but it is a distribution.

Our main theorem (Theorem \ref{main}) follows from two estimates about the pushforward measures $d^x_* \mu_1$ and $d^x_* \mug$.

\begin{proposition} \label{mainest1} If $\alpha > 1$, and if we choose $R_0$ large enough, then there is a subset $E_2' \subset E_2$ so that $\mu_2(E_2') \ge 1 - \frac{1}{1000}$ and for each $x \in E_2'$,
	
	$$	\| d^x_*(\mu_1) - d^x_*(\mug) \|_{L^1} < \frac{1}{1000}. $$
\end{proposition}

\begin{proposition} \label{mainest2} If $\alpha > 5/4$, then 
	
	$$	\int_{E_2}  \| d^x_*(\mug) \|_{L^2}^2 d \mu_2(x) < + \infty. $$
\end{proposition}

%Remark: I think that the value of the right hand side here depends only on the constants in the definition of $\mu_i$.

\begin{proof} [Proof of Theorem \ref{main} using Proposition \ref{mainest1} and Proposition \ref{mainest2}] The two propositions tell us that there is a point $x \in E_2$ so that 

\begin{equation} \label{L1close} \| d^x_*(\mu_1) - d^x_*(\mug) \|_{L^1} < 1/1000, \textrm{ and} \end{equation}

\begin{equation} \label{L2bound} \| d^x_*(\mug) \|_{L^2} < + \infty. \end{equation}  

Since $d^x_*(\mu_1)$ is a probability measure, (\ref{L1close}) guarantees that

$$ \int |d^x_*(\mug)| \ge 1 - \frac{1}{1000}. $$

\noindent Note that the support of $d^x_*(\mu_1)$ is contained in $\Delta_x(E)$.  Therefore

$$ \int_{\Delta_x(E)} | d^x_* \mug | = \int |d^x_* (\mug)| - \int_{\Delta_x(E)^c} |d^x_*(\mug)|$$

$$ \ge 1 - \frac{1}{1000} - \int |d^x_*(\mu_1) - d^x_*(\mug)| \ge 1 - \frac{2}{1000}. $$

But on the other hand, 

\begin{equation} \label{yes} \int_{\Delta_x(E)} | d^x_* \mug| \le | \Delta_x(E)|^{1/2} \left( \int |d^x_* \mug|^2 \right)^{1/2}. \end{equation}

Since (\ref{L2bound}) tells us that $\int |d^x_* \mug|^2$ is finite, it follows that $|\Delta_x(E)|$ is positive.  \end{proof}

To end this section, let us make some comments about the proofs of Proposition \ref{mainest1} and Proposition \ref{mainest2}.  To prove Proposition \ref{mainest1}, the first observation is that if $x$ is far from $T$, then removing $M_T \mu_1$ from $\mu_1$ has a negligible effect on the pushforward measure $d^x_* (\mu_1)$.  So the difference between $d^x_* \mu_1$ and $d^x_* \mug$ only comes from the bad tubes going through $x$.  Recall that a tube $T$ is bad if its $\mu_2$ measure is too large.  In general a point $x$ could lie in many bad tubes, and we need to know that the total $\mu_1$ measure of all these bad tubes is small (for most $x \in E_2$).  This follows from Orponen's radial projection theorem from \cite{O17b}.  This theorem plays an important role in Keleti and Shmerkin's work on the Falconer problem \cite{KS18}, which is where we learned about it.

To discuss Proposition \ref{mainest2}, we first describe the framework from \cite{Liu18}.  Let $\sigma_t$ denote the normalized arc length measure on the circle of radius $t$ (so the total measure is 1).  In \cite{Liu18}, Liu proved the following remarkable identity: for any function $f$,

$$ \int_0^\infty | f * \sigma_t (x)|^2 t dt = \int_0^\infty |f * \hat \sigma_r(x)|^2 r dr. $$

\noindent It follows from this identity that

$$ \int_{E_2} \| d^x_* (\mug) \|_{L^2}^2 d \mu_2(x) \lesssim \int_0^\infty \left( \int_{E_2} |\mug * \hat \sigma_r|^2 d\mu_2 (x) \right) r dr. $$

\noindent Now the Fourier transform of $\mug * \hat \sigma_r$ is supported on the circle of radius $r$, and studying such functions is the subject of restriction theory.  We can decompose $\mug * \hat \sigma_r$ as

$$ \mug * \hat \sigma_r = \sum_{T \textrm{ good}} M_T \mu_1 * \hat \sigma_r. $$

\noindent The right-hand side is essentially the wave packet decomposition of $\mug * \hat \sigma_r$.  This means that $M_T \mu_1 * \hat \sigma_r |_{B^2(1)}$ is essentially supported in $T$  and its Fourier transform is essentially supported in an arc of $S^1_r$ in the direction of $T$.  Since the tubes $T$ are all good, each tube $T$ has a small $\mu_2$ measure, and we will take advantage of this to bound the inner integral $\int_{E_2} | \mug * \hat \sigma_r |^2 d \mu_2(x)$.  

Since $T$ has a small $\mu_2$ measure, we can immediately get a good estimate for $\int_{E_2} | M_T \mu_1 * \hat \sigma_r|^2 d \mu_2$.  But to bound $\int_{E_2}| \mug* \hat \sigma_r|^2 d \mu_2$, we need to know how the wave packets $M_T \mu_1 * \hat \sigma_r$ interact with each other.  Is it possible that these wave packets have a lot of positive interference on the set $E_2$?  We will use decoupling theory to control such positive interference.  We will discuss this further in Section \ref{sec:RS}.

\section{Proof of Proposition \ref{mainest1}} \label{sec:mainest1}

We will study the pushforward measures $d^x_*(\mu_1)$ and $d^x_*(\mug)$.  Recall by definition that

$$ \int \psi(t) d^x_*(\mu) = \int \psi(|x-y|) d \mu(y). $$

In particular, if $\psi$ is the characteristic function of the interval $t_0 \le t \le t_0 + \Delta t$, then we see that

$$ \int_{t_0}^{t_0 + \Delta t} d^x_*(\mu) = \int_{t_0 \le |x-y| \le t_0 + \Delta t} d \mu. $$

To evaluate $d^x_*(\mu)$ at $t$, we can take the limit as $\Delta t \rightarrow 0$.  If we think of $\mu$ as $\mu(y) dy$, then we get

$$ d^x_*(\mu) (t) = \int_{S^1(x,t)} \mu(y) dl(y), $$

\noindent where $dl(y)$ denotes the arc length measure on the circle $S^1(x,t)$.  

To control $\| d^x_*(\mu_1) - d^x_*(\mug) \|_{L^1}$ we will start by studying $ d^x_*(M_T \mu_1)$ for different $T$.  For a tube $T \in \TT$, let $2T$ denote the concentric tube of twice the radius.  If $x \notin 2T$, we show that $d^x_*( M_T \mu_1)$ is negligible.

\begin{lemma} \label{xnotinT} If $T \in \TT_{j, \tau}$, and $x \in E_2$, and $x \notin 2T$, then
	
	$$ \| d^x_* (M_T \mu_1) \|_{L^1} \lesssim \RD(R_j). $$
\end{lemma}

\begin{proof} 
We will prove the stronger estimate that for every $t$:
	
	$$  d^x_* (M_T \mu_1) (t)  \lesssim \RD(R_j). $$

Recall that

\begin{equation} \label{dxMT} d^x_*  (M_T \mu_1)(t) = \int_{S^1(x,t)} M_T \mu_1(y) dl(y). \end{equation}

We also recall that 

$$ M_T \mu_1 = \eta_T (\psi_{j, \tau} \hat \mu_1)^\vee = \eta_T (\psi_{j, \tau}^\vee * \mu_1). $$

Now $\psi_{j, \tau}^\vee$ is concentrated on a $R_j^{-1/2} \times R_j^{-1}$ rectangle centered at 0 and it decays rapidly outside that rectangle.  Since $x \in E_2$, the distance from $x$ to the support of $\mu_1$ is $\gtrsim 1$.  Therefore, $d^x_* M_T \mu_1(t)$ is tiny unless $t \sim 1$.

To study the case when $t \sim 1$, we expand out $M_T \mu_1$:

$$ M_T \mu_1 (y) = \eta_T(y) ( \psi_{j, \tau} \hat \mu_1 )^{\vee}(y)  = \eta_T(y)  \int e^{2 \pi i \omega y} \psi_{j, \tau}(\omega) \hat \mu_1( \omega) d \omega. $$

Since $| \hat \mu_1(\omega) | \le 1$, and $\psi_{j, \tau}(\omega)$ is supported on $\tau$ and bounded by 1, it suffices to check that for each $\omega \in \tau$, 

\begin{equation} \label{statphase} \int_{S^1(x,t)} \eta_T(y) e^{2 \pi i \omega y} dl(y) \le \RD(R_j). \end{equation}

We will prove this rapid decay by stationary phase.  There are two slightly different cases, depending on whether $T$ intersects $S^1(x, t/2)$ or not.  Let us start with the case that $T$ intersects $S^1(x, t/2)$, since this case is a little harder.  After a coordinate rotation, we can assume that $\omega$ has the form $(0, \omega_2)$ with $\omega_2 \sim R_j$.  Recall that a tube $T \in \TT_{j, \tau}$ has long axis in the direction of the center of $\tau$. In particular, our tube $T$ must be nearly vertical, up to an angle of $R_j^{-1/2}$.   The tube $T$ intersects $S^1(x,t)$ in two arcs, which we deal with one at a time.  Each arc is a graph of the form $ y_2 = h(y_1)$, where $y_1$ lies in an interval $I(T)$ of length $\sim R_j^{-1/2 + \delta}$.   Since $T$ intersects $S^1(x,t/2)$, and the tube $T$ is nearly vertical, the function $h$ and all its derivatives are $\lesssim 1$ on $I(T)$.  

The following point is crucial for stationary phase.  Since $T$ is within an angle $R_j^{-1/2}$ of vertical, and $x \notin 2T$, then the distance from $T$ to the top or bottom points of the circle is $\gtrsim R_j^{-1/2 + \delta}$, and so we get 

$$ |h'(y_1)| \gtrsim R_j^{-1/2 + \delta} \textrm{ on the interval } I(T).$$

\noindent  In these coordinates, our integral becomes

$$ \int_{I(T)} \eta_T(y_1, h(y_1)) e^{2 \pi i \omega_2 h(y_1)} J(y_1) dy_1, $$

\noindent where $J(y_1)$ is the Jacobian factor that relates the arclength on the circle to $dy_1$. Since $h$ and all its derivatives are $\lesssim 1$ on $I(T)$, the same applies to $J$.  The function $\eta_T$ is smooth at scale $R_j^{-1/2 + \delta}$, and so if we abbreviate $\eta(y_1) := \eta_T(y_1, h(y_1)) J(y_1)$, then $\eta$ obeys 

$$ | \eta^{(k)} | \lesssim (R_j^{1/2 - \delta})^k ,$$

\noindent and $\eta$ is supported on $I(T)$.  We let $\phi(y_1) = 2 \pi \omega_2 h(y_1)$.  We now have to bound the following integral:

$$ \int_{I(T)} \eta(y_1) e^{ i \phi(y_1)} dy_1. $$

This integral can be bounded using stationary phase.  The method is essentially the same as in \cite{St}, Chapter 8, Proposition 1.  Here is a sketch.  We note that on $I(T)$, 

$$ | \phi'(y_1)| = |\omega_2| |h'(y_1)| \gtrsim R_j^{1/2 + \delta}, \textrm{ and }$$

$$ | \phi^{(k)} (y_1)| = | \omega_2| |h^{(k)}(y_1)| \lesssim R_j. $$

Next we note that

$$ \frac{1}{i \phi'} \frac{d}{dy_1} e^{i \phi(y_1)} = e^{ i \phi(y_1)}. $$

We define $D =  \frac{1}{i \phi'} \frac{d}{dy_1} $, so our integral becomes $\int_{I(T)} \eta D^N e^{i \phi} dy_1$, where $N$ is an arbitrary integer.  Next we expand out $D^N e^{i \phi}$ and we integrate by parts many times so that none of the derivatives actually lands on $e^{i \phi}$.  Using our lower bound on $|\phi'|$ and our upper bounds on the higher derivatives of $\phi$ and the derivatives of $\eta$, it follows that our integral is bounded by $C_N R_j^{- 2 \delta N}$.  For instance, if all the derivatives land on $\eta$, then we get a bound on $(R_j^{1/2-\delta})^N (R_j^{1/2 + \delta})^{-N} \sim R_j^{- 2 \delta N}$, and this is the worst case.  Since $N$ is arbitrary we get the desired bound.  

If $T$ does not intersect $S^1(x, t/2)$, then we choose our coordinates differently so that we can still arrange that $|h'|$ is bounded.  This time, we rotate so that $\omega = (\omega_1, 0)$, where $\omega_1 \sim R_j$.  The tube $T$ intersects $S^1(x,t)$ in one or two arcs, and each arc is a graph of the form $y_2 = h(y_1)$ over an interval $I(T)$, and $h$ and all its derivatives are $\lesssim 1$ on $I(T)$.  Our integral now has the form

$$ \int_{I(T)} \eta_T(y_1, h(y_1)) J(y_1) e^{2 \pi i \omega_1 y_1} dy_1. $$

\noindent Since $\omega_1 \sim R_j$, and $\eta_T$ is smooth on the scale $R_j^{-1/2 + \delta}$, this integral can also be bounded by stationary phase (in fact more simply than in the other case).  

\end{proof}

Next we prove a simple bound to cover the case that $x \in 2T$.

\begin{lemma} \label{MTfL1} For any $T \in \TT_{j, \tau}$ and any function $f$ supported in the unit disk,
	
	$$ \| M_T f \|_{L^1} \lesssim \| f \|_{L^1 (2T)} + \RD(R_j) \| f \|_{L^1}. $$ 
\end{lemma}

\begin{proof} Recall that for a tube $T \in \TT_{j, \tau}$, we defined $M_T$ by
		
	$$ M_T f := \eta_T (\psi_{j, \tau} \hat f)^{\vee} = \eta_T( \psi_{j, \tau}^\vee * f). $$
	
Now $\psi_{j, \tau}^\vee$ is essentially supported in a rectangle of dimensions $R_j^{-1} \times R_j^{-1/2}$ and $\| \psi_{j, \tau} \|_{L^1} \lesssim 1$.  Since the thickness of $T$ is $R_j^{-1/2 + \delta}$, we get

$$ \int |M_T f| \lesssim \int_{T} |\psi_{j, \tau}^\vee * f| \lesssim \int_{2T} |f| + \RD(R_j) \| f\|_{L^1}. $$
	
\end{proof}

\begin{corollary} \label{MTmuL1} For any point $x$, and any tube $T \in \TT$,
	
	$$\| d^x_* (M_T \mu_1) \|_{L^1} \lesssim \mu_1(2T) + \RD(R_j).$$
\end{corollary}

Next we check carefully that $\mu_1$ is very close to $M_0 \mu_1 + \sum_{T \in \TT} M_T \mu_1$.

\begin{lemma} \label{f=MTf} For any function $L^1$ function $f$ supported in the unit disk
	$$\| f - M_0 f - \sum_{T\in \TT} M_T f \|_{L^1} \lesssim \RD(R_0) \| f \|_{L^1}. $$ 
\end{lemma}

\begin{proof} Recall that $\{ \psi_{j, \tau} \}$ is a partition of unity.  We define
	$M_{j, \tau} f = (\psi_{j, \tau} \hat f)^{\vee}$, so that $f = \sum_{j, \tau} M_{j, \tau} f$.    It suffices to bound
	
	$$ \| M_{j, \tau} f - \sum_{T \in \TT_{j, \tau}} M_T f \|_{L^1} \lesssim \RD(R_j) \| f \|_{L^1}. $$
	
The left hand side is

$$ \| (1 - \sum_{T \in \TT_{j, \tau}} \eta_T) (\psi_{j, \tau}^\vee * f) \|_{L^1}. $$

Now as we noted in the proof of Lemma \ref{MTfL1}, $\psi_{j, \tau}^\vee$ is essentially supported on an $R_j^{-1/2} \times R_j^{-1}$ rectangle.  Also, $\sum_{T \in \TT_{j, \tau}} \eta_T$ is equal to one on the disk of radius 2 and then decays outside it.  Since $f$ is supported in the unit disk, $\psi_{j, \tau}^\vee * f$ is essentially supported in the disk of radius 2, and we get the desired rapid decay. \end{proof}

Now we can relate $ \| d^x_*(\mug) - d^x_*(\mu_1) \|_{L^1}$ to the geometry of the bad rectangles.  For each point $x$ and each $j$, we define

$$ \Bad_j(x) := \bigcup_{T \in \TT_j: x \in 2T \textrm{ and $T$ is bad}} 2T. $$

\begin{lemma} For any point $x$ in $E_2$,
	
	$$  \| d^x_*(\mug) - d^x_*(\mu_1) \|_{L^1} \lesssim \sum_{j \ge 1} \mu_1 ( \Bad_j(x)) + \RD(R_0). $$
\end{lemma}

\begin{proof} Recall that $\mug$ is defined by
	
	$$ \mug := M_0 \mu_1 + \sum_{T \in \TT, T \textrm{ good}} M_T \mu_1. $$
	
Using Lemma \ref{f=MTf}, we see that

	$$  \| d^x_*(\mug) - d^x_*(\mu_1) \|_{L^1} \lesssim \sum_{j} \sum_{T \in \TT_j, T \textrm{ bad}} \| d^x_*(M_T \mu_1) \|_{L^1} + \RD(R_0). $$
	
If $x \in 2T$, then we apply Corollary \ref{MTmuL1}, and if $x \notin 2T$, then we apply Lemma \ref{xnotinT}.  We get

	$$  \| d^x_*(\mug) - d^x_*(\mu_1) \|_{L^1} \lesssim \sum_{j} \sum_{T \in \TT_j, x \in 2T, T \textrm{ bad}} \mu_1(2T) + \RD(R_0). $$

Since the distance from $E_2$ to $E_1$ is $\gtrsim 1$, each point of $E_1$ is contained in $2T$ for $\lesssim 1$ tube $T \in \TT_j$ with $x \in 2T$.  Therefore, the right hand side is

$$ \lesssim \sum_j \mu_1 \left(\bigcup_{T \in \TT_j, x \in 2T, T \textrm{ bad}} 2T \right)+ \RD (R_0) = \sum_j \mu_1( \Bad_j(x) ) + \RD(R_0). $$

\end{proof}

Next we need to estimate the measure of $\Bad_j(x)$.  We will do this using Orponen's radial projection theorem.  Before introducing the theorem, we need to set up a little more notation.

$$ \Bad_j := \{ (x_1, x_2): \textrm{ there is a bad } T \in \TT_j \textrm{ so that } 2T \textrm{ contains } x_1 \textrm{ and } x_2 \}. $$

Notice that $\Bad_j(x)$ is just the set of $y$ so that $(y,x) \in \Bad_j$.  Therefore, 

$$ \mu_1 \times \mu_2 (\Bad_j) = \int \mu_1(\Bad_j(x)) d\mu_2(x). $$

Our main estimate about the bad rectangles is

\begin{lemma} \label{badrect} For each $\alpha > 1$, there is a constant $c(\alpha) > 0$ so that for each $j \ge 1$,
	$$ \mu_1 \times \mu_2(\Bad_j) \lesssim R_j^{- c(\alpha) \delta}. $$
\end{lemma}

Before turning to the proof, let us use this lemma to finish the proof of Proposition \ref{mainest1}. 

\begin{proof} [Proof of Proposition \ref{mainest1} using Lemma \ref{badrect}] We want to find a set $E_2' \subset E_2$ with $\mu_2(E_2') \ge 1 - \frac{1}{1000}$ so that for each $x \in E_2'$,
	
	$$   \| d^x_*(\mug) - d^x_*(\mu_1) \|_{L^1} \le \frac{1}{1000}. $$
	
	We recall that
	
	$$ \mu_1 \times \mu_2 (\Bad_j) = \int \mu_1(\Bad_j(x)) d\mu_2(x). $$
	
Therefore, we can choose $B_j \subset E_2$ so that $ \mu_2(B_j) \le R_j^{- (1/2) c(\alpha) \delta}$ and for all $x \in E_2 \setminus B_j$, 

$$ \mu_1(\Bad_j(x)) \lesssim R_j^{- (1/2) c(\alpha) \delta}. $$

We define $E_2' = E_2 \setminus \bigcup_{j \ge 1} B_j$.  As long as $R_0$ is sufficiently large (compared to $\delta$ and $\alpha$), we have $\mu_2(E_2') \ge 1 - \frac{1}{1000}$ as desired.  Now for each $x \in E_2'$, we have

$$   \| d^x_*(\mug) - d^x_*(\mu_1) \|_{L^1} \lesssim \sum_{j \ge 1} \mu_1 ( \Bad_j(x)) + \RD(R_0) \lesssim $$

$$ \sum_{j \ge 1} R_j^{-(1/2) c(\alpha) \delta} \lesssim R_0^{- (1/2) c(\alpha) \delta}. $$

By choosing $R_0$ sufficiently large, we get the desired bound.

\end{proof}

Now we introduce Orponen's radial projection theorem.  The statement we use appears as Proposition 3.11 in \cite{KS18}, and it appears as Equation (3.5) in Orponen's paper \cite{O17b}.  Define a radial projection map $P_y: \RR^2 \setminus \{ y \} \rightarrow S^1$ by

$$P_y(x) = \frac{x-y}{|x-y|}. $$

\begin{theorem} \label{orponenthm}(Orponen, \cite{O17b}) For every $\alpha > 1$ there exists $p(\alpha)> 1$ so that the following holds.  Suppose that $\mu_1$ and $\mu_2$ are measures on the unit disk with disjoint supports and that for every ball $B(x,r)$, $\mu_i (B(x, r)) \lesssim r^\alpha$.  Then
	
	$$ \int \| P_y \mu_2 \|_{L^p}^p d \mu_1(y) < + \infty. $$

\end{theorem}

\begin{proof}[Proof of Lemma \ref{badrect}]

Recall that $\Bad_j(y)$ is defined to be 

$$ \Bad_j(y) := \bigcup_{T \in \TT_j: y \in 2T \textrm{ and $T$ is bad}} 2T. $$

\noindent In other words, $\Bad_j(y)$ is the set of $x$ so that $(y,x)$ lies in $\Bad_j$.  Therefore,

$$ \mu_1 \times \mu_2( \Bad_j) = \int \mu_2(\Bad_j(y)) d \mu_1(y). $$

Suppose that $T \in \TT_j$ is a bad rectangle and $y \in 2T$.  Let $A(T)$ be the arc of the circle whose center corresponds to the direction of the long axis of $T$ and with length $\sim  R_j^{-1/2 + \delta}$.  Since the distance from $E_1$ to $E_2$ is $\gtrsim 1$, it follows that $P_y (2T) \subset A(T)$, and so

\begin{equation} \label{bigarc} P_y \mu_2 (A(T)) \ge \mu_2(2T) \ge R_j^{-1/2 + 100 \delta}.  \end{equation}

 So we see that $P_y( \Bad_{j}(y))$ can be covered by arcs $A(T)$ of length $\sim R_j^{-1/2 + \delta}$ which each enjoy (\ref{bigarc}).  By the Vitali covering lemma, we can choose a disjoint subset of the arcs $A(T)$ so that $5 A(T)$ covers $P_y( \Bad_{j}(y))$.  This implies that the arc length measure of $P_y( \Bad_{j}(y))$ is bounded by 

$$ | P_y(\Bad_{j}(y)) | \lesssim R_j^{-99 \delta}. $$

Now we bound

$$ \mu_1 \times \mu_2(\Bad_j) = \int \mu_2 (\Bad_{j}(y)) d\mu_1(y)  \le \int \left( \int_{P_y (\Bad_{j}(y)) } P_y \mu_2 \right) d \mu_1(y).$$
	
By Holder's inequality, this is	
	
$$ \le  | P_y(\Bad_{j}(y)) |^{1 - \frac{1}{p}} \int \| P_y \mu_2 \|_{L^p} d \mu_1 \lesssim R_j^{- c(\alpha) \delta}. $$

\end{proof}

\section{Refined Strichartz estimates}\label{sec:RS}

The proof of Proposition \ref{mainest2} will use a refined Strichartz type estimate, which in turn is based on the decoupling theorem of Bourgain-Demeter \cite{BD15}.  

\begin{theorem} \label{dec} (\cite{BD15}) 
	Suppose that $S \subset \RR^d$ is a strictly convex $C^2$ hypersurface with Gaussian curvature $\sim 1$.   Decompose the $R^{-1}$-neighborhood of $S$ into blocks $\theta$ of dimensions $R^{-1/2} \times ... \times R^{-1/2} \times R^{-1}$.  Suppose that $\hat f_\theta$ is supported in $\theta$ and $f = \sum_{\theta} f_\theta$.  Then for any $p$ in the range $2 \le p \le \frac{2(d+1)}{d-1}$, 
	
\begin{equation} \label{eqdec} \| f \|_{L^p(B_R)} \lessapprox \left(\sum_\theta \|  f_\theta \|_{L^p(w_{B_R})}^2 \right)^{1/2}, \end{equation}
	
	\noindent where $w_{B_R}$ is a weight which is $\sim 1$ on $B_R$ and rapidly decaying.
\end{theorem}

The decoupling theorem is a remarkably strong and sharp theorem in some situations, for instance if $|f_{\theta}(x)|$ is roughly constant on $B_R$ for each $\theta$.  On the other hand, if the supports of the different $f_\theta$ are disjoint from each other, then one trivially gets the stronger inequality $\| f \|_{L^p(B_R)} \le ( \sum_\theta \| f_\theta \|_{L^p(B_R)}^p)^{1/p}$.  The idea of refined Strichartz estimates is to use the decoupling theorem where it is strong, but also to take advantage of disjointness when it occurs.  The first version of the refined Strichartz inequality appeared in \cite{DGL17}, and it was generalized in \cite{DGLZ18}.  We need here a slightly more flexible version of the inequality.  The inequality we prove here was discovered independently by Xiumin Du and Ruixiang Zhang (personal communication).

We will state our estimate in terms of wave packets.  Here is the setup.  Let $S$ and $\theta$ be as above.  Let $\TT_\theta$ be a finitely overlapping covering of $\RR^d$ by tubes $T$ of length $\sim R^{1+\delta}$ and radius $\sim R^{\frac{1+\delta}{2}}$ with long axis normal to the surface $S$ at $\theta$.  We write $\TT = \cup_\theta \TT_\theta$.  Each $T \in \TT$ belongs to $\TT_\theta$ for a single $\theta$, and we let $\theta(T)$ denote this $\theta$.  We say that $f$ is microlocalized to $(T, \theta(T))$ if $f$ is essentially supported in $2T$ and $\hat f$ is essentially supported in $2 \theta(T)$.  A function $f_T$ which is microlocalized to $(T, \theta(T))$ is called a wave packet.  If $\omega \in \theta(T)$, then $f_T$ morally has the form $ f_T \approx a \chi_T e^{2 \pi i \omega x}$, where $a \in \CC$ and $\chi_T$ denotes a smooth bump function on $T$.  Our theorem gives an estimate for the constructive interference between wave packets.

\begin{theorem} \label{refdec} Let $p$ be in the range $2 \le p \le \frac{2(d+1)}{d-1}$.  Let $\WW \subset \TT$ and suppose that each $T \in \WW$ lies in $B_R$.  Let $W = | \WW|$.  Suppose that $f = \sum_{T \in \WW} f_T$, where $f_T$ is microlocalized to $(T, \theta(T))$.  Suppose that $\| f_T \|_{L^p}$ is roughly constant among all the $T \in \WW$.   Let $Y$ be a union of $R^{1/2}$-cubes in $B_R$ each of which intersects at most $M$ tubes $T \in \WW$.  Then
	
	\begin{equation} \label{eqrefdec} \| f \|_{L^p(Y)} \lesssim R^\epsilon \left(\frac{M}{W} \right)^{\frac{1}{2} - \frac{1}{p}} \left(\sum_{T \in \WW} \| f_T \|_{L^p}^2 \right)^{1/2}. \end{equation}
	
\end{theorem}

The fraction $M/ W$ measures to what extent the wave packets of $\WW$ are disjoint from each other.  If $M=1$, then the wave packets are completely disjoint, and the inequality above becomes $\| f \|_{L^p(Y)} \lessapprox ( \sum_T \| f_T \|_{L^p}^p)^{1/p}$.

Before proving Theorem \ref{refdec}, let us explain how it relates to the decoupling theorem (Theorem \ref{dec}).  In Theorem \ref{dec}, consider the special case that $f_\theta$ is non-zero for $N$ caps $\theta$, and that for each of these caps, $f_\theta = \sum_{T \in \TT_\theta} f_T$ is a sum of $P$ non-zero wave packets $f_T$, and that all these wave packets have the same amplitude.  In \cite{BD15}, the general theorem was reduced to this special case by pigeonholing, so it is not really so special.  In this case, the decoupling inequality (\ref{eqdec}) can be written in the form

\begin{equation}\label{eqdecex} \| f \|_{L^p(B_R)} \lessapprox \left(\frac{1}{P}\right)^{\frac{1}{2} - \frac{1}{p}}  \left(\sum_{T \in \WW} \| f_T \|_{L^p}^2 \right)^{1/2}. \end{equation}

\noindent Now if $Q$ is any $R^{1/2}$-square, then it can lie in $\lesssim 1$ tube $T$ in each direction.  Therefore, we have $M \le N$, and $W = NP$, and so $\frac{M}{W} \le \frac{1}{P}$.  So we see that (\ref{eqrefdec}) is at least as strong as (\ref{eqdecex}), and it is stronger whenever $M$ is much less than $N$.  When $M$ is much less than $N$, then it means that each cube $Q$ lies in wave packets from only a small fraction of the different caps $\theta$, which means that the supports of the $f_\theta$ don't intersect as much as they could.  In summary, Theorem \ref{refdec} is like Theorem \ref{dec}, but it gives a stronger estimate when the supports of the $f_\theta$ don't intersect too much.

\begin{proof} [Proof of Theorem \ref{refdec}] 
	Without loss of generality, we can assume that 
	
	\begin{equation} \label{Qconst} \|  f \|_{L^p(Q)} \sim \textrm{ constant for all $R^{1/2}$-cubes $Q \subset Y$}. \end{equation}
	
	To set up the argument, we decompose $f$ as follows.  We cover $S$ with larger blocks $\tau$ of dimensions $R^{-1/4} \times ... \times R^{-1/4} \times R^{-1/2}$.  
	For each $\tau$ we cover $B^d(R)$ with cylinders $\Box$ with radius $R^{3/4}$ and length $R$, with the long axis perpendicular to $\tau$.  Each cylinder $\Box$ is associated to one $\tau$, which denote $\tau(\Box)$.  Then we define
	
	$$\WW_\Box := \{ T \in \WW: \theta(T) \subset \tau(\Box) \textrm{ and } T \subset \Box \}. $$
	
	We define $f_\Box = \sum_{T\in \WW_\Box} f_T. $  We note that $\hat f_\Box$ is essentially supported in $\tau(\Box)$.   An $R^{1/2}$-cube $Q$ lies in one cylinder $\Box$ associated to each cap $\tau$.  So by applying decoupling at scale $R^{1/2}$, we get 
	
	\begin{equation} \label{decex} \| f \|_{L^p(Q)} \lessapprox \left( \sum_{\Box} \| f_{\Box} \|_{L^p(Q)}^2 \right)^{1/2}. \end{equation}
	
	(Stricly speaking, we have a weight on the right-hand side.  However, if the tail of the weight dominates for some $Q \subset Y$, then we trivially get the conclusion of the theorem.  Therefore, we can ignore the tail of the weight.)
	
	The next ingredient is induction on scales.  After parabolic rescaling, the decomposition $f_\Box = \sum_{T \in \WW_\Box} f_T$ is equivalent to the setup of the theorem at scale $R^{1/2}$ instead of scale $R$.  So by induction on the radius, we get a version of our main inequality for each function $f_\Box$.   It goes as follows:
	
	Write $\Box$ as a union of $R^{1/2} \times R^{3/4}$ cylinders running parallel to the long axis of $\Box$.  Let $Y_{\Box, M'}$ be the union of those cylinders that intersect $\sim M'$ of the tubes $T \in \WW_\Box$.  Then
	
	\begin{equation}\label{indbox} \| f_\Box \|_{L^p(Y_{\Box, M'})} \lesssim R^{\epsilon/2}  \left(\frac{M'}{| \WW_\Box|} \right)^{\frac{1}{2} - \frac{1}{p}} \left(\sum_{T \in \WW_\Box}   \| f_T \|_{L^p}^2 \right)^{1/2}. \end{equation}
	
	Now we dyadically pigeonhole $M'$ so that
	
	$$ \| f \|_{L^p(Q)} \lessapprox \left\| \sum_{\Box: Q \subset Y_{\Box, M'}} f_\Box \right\|_{L^p(Q)} $$
	
	\noindent for a fraction $\approx 1$ of $Q \subset Y$.
	
	We fix this value of $M'$, and from now on we abbreviate $Y_\Box = Y_{\Box, M'}$.  
	
	Next we dyadically pigeonhole $| \WW_\Box|$.  Let $\mathbb{B}_{W'}$ be the set of $\Box$ with $| \WW_\Box | \sim W'$.  We dyadically pigeonhole $W'$ so that
	
	\begin{equation} \label{decentcube} \| f \|_{L^p(Q)} \lessapprox \left\| \sum_{\Box \in \mathbb{B}_{W'}: Q \subset Y_\Box} f_\Box \right\|_{L^p(Q)}. \end{equation}
	
	\noindent for a fraction $\approx 1$ of $Q \subset Y$.
	
	We fix this value of $W'$ and from now on we abbreviate $\mathbb{B} = \mathbb{B}_{W'}$.
	
	We also note that for each $\Box \in \mathbb{B}$,
	
	\begin{equation} \label{fboxl2} \sum_{T \in \WW_\Box} \| f_T \|_{L^p}^2 \sim \frac{W'}{W} \sum_{T \in \WW} \| f_T \|_{L^p}^2. \end{equation} 
	
	Finally, we dyadically pigeonhole the cubes $Q \subset Y$ according to the number of $\Box \in \mathbb{B}$ so that $Q \subset Y_\Box$.  We get a subset $Y' \subset Y$ so that for each cube $Q \subset Y'$, $Q \subset Y_\Box$ for $\sim M''$ choices of $\Box \in \mathbb{B}$, and $Q$ obeys (\ref{decentcube}).  Moreover, by dyadic pigeonholing, we have $|Y'| \approx |Y|$.  Since each cube $Q \subset Y$ had approximately equal $L^p$ norm, we also get $\| f \|_{L^p(Y')} \approx \| f \|_{L^p(Y)}$.
	
	We also note that
	
	$$ M'  M'' \le M. $$
	
	\noindent because a cube $Q \subset Y'$ belongs to $Y_\Box$ for $\sim M''$ different $\Box$, and if $Q \subset Y_\Box$, then it belongs to $T$ for $\sim M'$ different $T \in \WW_\Box$.
	
	Similarly, we note that
	
	$$ W' | \mathbb{B}| \le W $$
	
	\noindent because for each $\Box \in \mathbb{B}$, $| \WW_\Box| \sim W'$, and $\WW_\Box$ are disjoint subsets of $\WW$.  
	
	Now we are ready to begin our estimate.  For each $Q \subset Y'$, we have
	
	$$ \| f \|_{L^p(Q)} \lessapprox  \left\| \sum_{\Box \in \mathbb{B}: Q \subset Y_{\Box}} f_\Box \right\|_{L^p(Q)}. $$
	
	Applying decouping as in (\ref{decex}), this is bounded by
	
	$$ \lessapprox  \left(\sum_{\Box \in \mathbb{B}: Q \subset Y_{\Box}}  \| f_\Box \|_{L^p(Q)}^2 \right)^{1/2}. $$
	
	The number of terms in the sum is $\sim M''$.  Applying H\"older, we get
	
	$$ \lesssim (M'')^{\frac{1}{2} - \frac{1}{p}}  \left(\sum_{\Box \in \mathbb{B}: Q \subset Y_{\Box}}  \| f_\Box \|_{L^p(Q)}^p \right)^{1/p}. $$
	
	We raise this inequality to the $p^{th}$ power and sum over $Q \subset Y'$ to get
	
	$$ \| f \|_{L^p(Y)}^p \lessapprox \| f \|_{L^p(Y')}^p \lessapprox (M'')^{\frac{p}{2} - 1} \sum_{\Box \in \mathbb{B}} \| f_\Box \|_{L^p(Y_\Box)}^p. $$
	
	Now we can use our induction on scales -- equation (\ref{indbox}) -- which gives
	
	$$ \lesssim R^{p\epsilon/2}  \left(\frac{M' M''}{W'} \right)^{\frac{p}{2} - 1} \sum_{\Box \in \mathbb{B}} \left(\sum_{T \in \WW_\Box} \| f_T \|_{L^p}^2 \right)^{p/2}. $$
	
	By (\ref{fboxl2}), this is 
	
	$$ \lesssim  R^{p\epsilon/2}  \left(\frac{M' M''}{W} \right)^{\frac{p}{2} - 1} \frac{|\mathbb{B}| W'} {W}\left( \sum_{T \in \WW}  \| f_T \|_{L^p}^2 \right)^{p/2}. $$
	
	Since $M'  M'' \le M$ and $|\mathbb{B}| W' \le W$, we get
	
	$$ \lesssim  R^{p\epsilon/2}  \left(\frac{ M}{W} \right)^{\frac{p}{2} - 1} \left( \sum_{T \in \WW}  \| f_T \|_{L^p}^2 \right)^{p/2}.$$
	
	Putting everything together and taking account of $\lessapprox$ throughout, we get
	
	$$ \| f \|_{L^p(Y)} \lesssim R^{3 \epsilon/4} 	 \left(\frac{ M}{W} \right)^{\frac{1}{2} - \frac{1}{p}} \left( \sum_{T \in \WW}  \| f_T \|_{L^p}^2 \right)^{1/2}.$$
	
	This closes the induction and finishes the proof.	 \end{proof}

One can also apply a rescaling to this theorem.  If we rescale in Fourier space by a factor $\lambda$, then each $R^{-1/2} \times ... \times R^{-1}$ block $\theta$ is replaced by a $\lambda R^{-1/2} \times ... \times \lambda R^{-1}$ block.  There is a corresponding rescaling in physical space so that each $R^{1/2} \times ... \times R$ tube $T$ is replaced by a 
$\lambda^{-1} R^{1/2} \times ... \times \lambda^{-1} R$ tube $T$.  The case of interest for us is $\lambda = R$.

\begin{corollary} \label{correfdec} 	Suppose that $S \subset \RR^d$ is a strictly convex $C^2$ hypersurface with Gaussian curvature $\sim 1$.    Suppose that the 1-neighborhood of $R S$ is partitioned into $R^{1/2} \times ... \times R^{1/2} \times 1$ blocks $\theta$.  For each $\theta$, let $\TT_\theta$ be a set of tubes of dimensions $R^{-1/2 + \delta} \times 1$ with long axis perpendicular to $\theta$, and let $\TT = \cup_\theta \TT_\theta$. 
	
	Let $p$ be in the range $2 \le p \le \frac{2(d+1)}{d-1}$.  Let $\WW \subset \TT$ and suppose that each $T \in \WW$ lies in the unit ball.  Let $W = | \WW|$.  Suppose that $f = \sum_{T \in \WW} f_T$, where $f_T$ is microlocalized to $(T, \theta(T))$.  Suppose that for each $T \in \WW$, $\| f_T \|_{L^p}$ is roughly constant.   Let $Y$ be a union of $R^{-1/2}$-cubes in $B_R$ each of which intersects at most $M$ tubes $T \in \WW$.  Then
	
	$$ \| f \|_{L^p(Y)} \lesssim R^\epsilon \left(\frac{M}{W} \right)^{\frac{1}{2} - \frac{1}{p}} \left(\sum_{T \in \WW} \| f_T \|_{L^p}^2 \right)^{1/2}. $$
	
\end{corollary}

Corollary \ref{correfdec} is the result we will actually use in our estimates about the Falconer problem.  

Theorem \ref{refdec} is closely related to the refined Strichartz estimates from \cite{DGL17}, \cite{DGLZ18} and \cite{DGOWWZ18}, and we record a corollary in a similar form.  To set up the statement, we need to set up a little notation.  We find it most convenient to work with the case that $S$ is a graph, so suppose $S$ is defined by $\omega_d = \phi(\omega_1, ..., \omega_{d-1})$ , and $(\omega_1, ..., \omega_{d-1}) \in B^{d-1}(1)$.  We assume that $\phi$ is $C^2$ and that the eigenvalues of the Hessian $\nabla^2 \phi$ are $\sim 1$.  Then for a function $g: B^{d-1} \rightarrow \CC$, we can define the extension operator by

\begin{equation} \label{defE} E g(x) = \int_{B^{d-1}} e^{2 \pi i (x_1 \omega_1 + ... + x_{d-1} \omega_{d-1}+ x_d \phi)} g(\omega_1, ..., \omega_{d-1})\,d\omega_1\cdots d\omega_{d-1} . \end{equation}

We decompose $B^{d-1}$ into finitely overlapping balls $\theta$ of radius $\sim R^{-1/2}$, and then we can decompose $g$ as

$$ g = \sum_{\theta, v} g_{\theta,v}, $$

where

\begin{enumerate}
	
	\item $v \in R^{1/2 + \delta} \ZZ^{d-1}$
	
	\item $g_{\theta,v}$ is supported on $\theta$.
	
	\item $\hat g_{\theta, v}$ is essentially supported on a ball around $v$ of radius $R^{1/2 + \delta}$.
	
	\item Therefore, the functions $g_{\theta,v}$ are approximately orthogonal.
	
	\item $E g_{\theta, v}$ restricted to $B_R$ is essentially supported on a tube $T_{\theta,v}$ of radius $\sim R^{1/2 + \delta}$ and length $\sim R$.
	
	\item If we think of $\theta$ as a cap in $S$, then the long axis of $T_{\theta, v}$ is normal to $S$.  Also $T_{\theta, v}$ intersects the plane $x_d = 0$ at the point $(v,0)$.  	
\end{enumerate}

See Section 3 of \cite{GuthII} for background on this wave packet decomposition, including proofs of these standard facts.

Now we are ready to state our refined Strichartz estimate.

\begin{theorem} \label{RS} Let $E$ be the extension operator as in \eqref{defE}, where $\phi$ is $C^2$ and the eigenvalues of the Hessian $\nabla^2 \phi$ are $\sim 1$.  Suppose that $g: B^{d-1} \rightarrow \CC$.  Suppose that $g = \sum_{(\theta, v) \in \WW} g_{\theta, v}$, where $\| g_{\theta, v} \|_{L^2}$ are comparable for all $(\theta, v) \in \WW$.  Let $W = | \WW|$.   Suppose $Y$ is a union of $R^{1/2}$-cubes in $B_R^d$ which each intersect $\sim M$ of the tubes $T_{\theta,v} \in \WW$.  Suppose that $p = \frac{2(d+1)}{d-1}$.  Then
	
	$$ \| Eg \|_{L^p(Y)} \lesssim R^\epsilon \left(\frac{M}{W} \right)^{\frac{1}{2} - \frac{1}{p}}  \| g \|_{L^2}. $$
\end{theorem}

\begin{proof} Let $\eta_{B_R}$ be a bump function associated to the ball of radius $R$.  We define
	
	$$ f_{\theta, v} = \eta_{B_R} Eg_{\theta, v}. $$
	
	The function $f_{\theta, v}$ is essentially supported in $T_{\theta, v}$ and its Fourier transform is essentially supported in the $R^{-1}$-neighborhood of $\theta$ (viewing $\theta$ as a cap in $S \subset \RR^d$).  Therefore, the functions $f_{\theta,v}$ have the right microlocalization to apply Theorem \ref{refdec}.  Before doing so, we need to sort them by $L^p$-norm.  We define
	
	$$ \WW_{\lambda} := \{ (\theta, v) \in \WW: \| f_{\theta,v} \|_{L^p} \sim \lambda \}. $$
	
	We define $g_\lambda := \sum_{(\theta, v) \in \WW_\lambda} g_{\theta, v}$, and $W_\lambda = | \WW_\lambda|$.  Now Theorem \ref{refdec} gives
	
	$$ \| E g_\lambda \|_{L^p(Y)} \lessapprox \left(\frac{M}{W_\lambda} \right)^{\frac{1}{2} - \frac{1}{p}} \left(\sum_{(\theta,v) \in \WW_\lambda} \| Eg_{\theta,v} \|_{L^p(B_R)}^2 \right)^{1/2}. $$
	
	Next we note that $\| E g_{\theta,v} \|_{L^p(B_R)} \lesssim \| g_{\theta, v} \|_{L^2}$.  This is a consequence of the Strichartz or Tomas-Stein inequality, but because $E g_{\theta, v}$ is a single wave packet, there is an even simpler argument:
	
	$$ \| E g_{\theta, v} \|_{L^p(B_R)} \lesssim \| E g_{\theta, v} \|_{L^p(T_{\theta, v})} \le | T_{\theta, v} |^{1/p} \| E g_{\theta,v} \|_{L^\infty} \le $$
	
	$$ \le  | T_{\theta, v}|^{1/p} \int_{\theta} | g_{\theta, v}| \le  | T_{\theta, v}|^{1/p} |\theta|^{1/2} \| g_{\theta, v} \|_{L^2}. $$
	
	\noindent Now $T_{\theta, v}$ has volume $R^{\frac{1}{2}(d-1) + 1}$ and $\theta$ has volume $R^{-\frac{d-1}{2}}$ and so $| T_{\theta,v}|^{1/p} |\theta|^{1/2} \lesssim 1$.  Plugging in this bound, we get
	
	$$  \| E g_\lambda \|_{L^p(Y)} \lessapprox \left(\frac{M}{W_\lambda} \right)^{\frac{1}{2} - \frac{1}{p}} \left(\sum_{(\theta, v) \in \WW_\lambda} \| g_{\theta, v} \|_{L^2}^2 \right)^{1/2}. $$
	
	Since all the $\| g_{\theta, v} \|_{L^2}$ are comparable, we get
	
	$$  \| E g_\lambda \|_{L^p(Y)} \lessapprox \left(\frac{M}{W_\lambda} \right)^{\frac{1}{2} - \frac{1}{p}} \left(\frac{W_\lambda}{W} \right)^{1/2} \| g \|_{L^2}.  $$
	
	We have $W_\lambda \le W$, and the total power of $W_\lambda$ on the right-hand side is positive, and so we get the bound
	
	$$  \| E g_\lambda \|_{L^p(Y)} \lessapprox \left(\frac{M}{W} \right)^{\frac{1}{2} - \frac{1}{p}}  \| g \|_{L^2}.  $$
	
	Since this estimate holds for every $\lambda$, the theorem is proven. \end{proof}

\section{Proof of Proposition \ref{mainest2}} \label{sec:mainest2}

In this section, we prove Proposition \ref{mainest2}.  The proof is based on adding a refined Strichartz estimate (Corollary \ref{correfdec}) to the framework of \cite{Liu18}.  We want to show that if $\alpha > 5/4$, then 

		$$	\int_{E_2}  \| d^x_*(\mug) \|_{L^2}^2 d \mu_2(x) < + \infty. $$

We follow Liu's approach from \cite{Liu18}.  Let $\sigma_t$ be the normalized arc length measure on the circle of radius $t$ (normalized so that the total measure is 1).  Then

$$ \| d^x_*(\mug) \|_{L^2}^2 = \int_0^\infty | \mug * \sigma_t (x)|^2 t^2 dt. $$

Now we would like to make use of Liu's identity:

\begin{theorem} \label{Liu} (\cite{Liu18}) For any function $f: \RR^2 \rightarrow \CC$, and any $x \in \RR^2$, 
	
	$$  \int_0^\infty | f * \sigma_t (x)|^2 t dt = \int_0^\infty | f * \hat \sigma_r (x)|^2 r dr.  $$
	
\end{theorem}

Notice that on the left-hand side we have $t dt$ instead of $t^2 dt$.  If $x \in E_2$, then $ \mu_1 * \sigma_t(x) = 0$ unless $t \sim 1$ because $E_1$ and $E_2$ are contained in the unit disk and the distance between them is $\gtrsim 1$.  Therefore, we can write

$$ \int_0^\infty | \mu_1 * \sigma_t(x)|^2 t^2 dt \sim \int_0^\infty | \mu_1 * \sigma_t(x)|^2 t dt. $$

We would like to write the same thing with $\mug$ in place of $\mu_1$.  To justify this, we need to argue that $\mug$ is essentially supported in a small neighborhood of $E_1$, which we now check.

\begin{lemma} Let $A$ be the complement of the $R_0^{-1/2 + \delta}$-neighborhood of $E_1$.  Then
	
	$$ \int_{A} |\mug| = \RD(R_0) \textrm{ and } \max_{x \in A} |\mug(x)| = \RD(R_0). $$
\end{lemma}

\begin{proof} By definition,
	
	$$ \mug = M_0 \mu_1 + \sum_{j, \tau} \sum_{\substack{T\in \TT_{j,\tau}\\T \textrm{ good }}} M_T \mu_1 = $$
	
	$$ = \psi_0^\vee * \mu_1 + \sum_{j, \tau} \sum_{T} \eta_T (\psi_{j, \tau}^\vee * \mu_1). $$
	
Now $\psi_0^{\vee}$ is essentially supported on a ball of radius $R_0^{-1}$ and $\psi_{j, \tau}^\vee$ is essentially supported on a rectangle of dimensions $R_j^{-1/2} \times R_j^{-1}$ centered at the origin.  Since $\mu_1$ is supported on $E_1$, the result follows.
\end{proof}

Since $\mug$ is essentially supported in a thin neighborhood of $E_1$, we can indeed say that for any $x \in E_2$,

$$ \int_0^\infty | \mug * \sigma_t(x)|^2 t^2 dt \lesssim \int_0^\infty | \mug * \sigma_t(x)|^2 t dt. $$

Now we can apply Theorem \ref{Liu} to get

$$ 	\int_{E_2}  \| d^x_*(\mug) \|_{L^2}^2 d \mu_2(x) \lesssim \int_{E_2}  \int_0^\infty | \mug * \hat \sigma_r (x)|^2 r dr d\mu_2(x) = $$

\begin{equation} \label{liuint} = \int_0^\infty \left(\int_{E_2} | \mug * \hat \sigma_r(x)|^2 d \mu_2(x) \right) r dr.  \end{equation}

We will use Theorem \ref{refdec} to estimate the inner integral for each $r$.

\begin{proposition} \label{mug*sr} For any $\alpha > 0$, $r>0$:
	
	$$ \int_{E_2} | \mug * \hat \sigma_r(x)|^2 d \mu_2(x) \le C(R_0) r^{- \frac{\alpha + 1}{3} + \epsilon} r^{-1} \int  | \hat \mu_1 |^2 \psi_r d \xi, $$

\noindent where $\psi_r$ is a weight function which is $\sim 1$ on the annulus $r-1 \le |\xi| \le r+1$ and decays off of it.  To be precise, we could take

$$ \psi_r(\xi) = \left( 1 + | r- |\xi|| \right)^{-100}.$$

\end{proposition}

The conclusion here is very similar to saying 

$$\int_{E_2}| \mug * \hat \sigma_r(x)|^2 d \mu_2(x) \le C(R_0) r^{-\frac{\alpha +1}{3} + \epsilon} \| \hat \mu_1 \|_{L^2(d \sigma_r)}^2.$$  

\noindent For technical reasons, we have the bound in the form above.   Before turning to the proof of Proposition \ref{mug*sr}, let us see how it implies Proposition \ref{mainest2}.  Like most previous work on the Falconer problem, the proof uses the idea of the $\beta$-dimensional energy of a measure.  Recall that this energy is given by

$$ I_\beta (\mu) := \int |x-y|^{-\beta} \mu(x) \mu(y). $$

If a measure $\mu$ on the unit ball obeys $\mu (B(x,r)) \lesssim r^\alpha$, then $I_\beta(\mu)$ is finite for every $\beta < \alpha$  (cf. Lemma 8.3 of \cite{W03}).  In particular, $I_\beta(\mu_1) < \infty$ for every $\beta < \alpha$.  There is also a Fourier representation for $I_\beta(\mu)$ (cf. Proposition 8.5 of \cite{W03}): if $\mu$ is a measure on $\RR^n$, then 

$$ I_{\beta} (\mu) = c_{n, \beta}  \int_{\RR^n} |\xi|^{-(n - \beta)} |\hat \mu(\xi) |^2 d \xi. $$

\begin{proof} [ Proof of Proposition \ref{mainest2} using Proposition \ref{mug*sr}] By (\ref{liuint}),
	
\begin{equation} \label{nirvanaisnear} 	\int_{E_2}  \| d^x_*(\mug) \|_{L^2}^2 d \mu_2(x) \lesssim  \int_0^\infty \left(\int_{E_2} | \mug * \hat \sigma_r(x)|^2 d \mu_2(x) \right) r dr. \end{equation}

Plugging in Proposition \ref{mug*sr} to bound the inner integral, we get

$$ \lesssim_{R_0} \int_0^\infty \int_{\RR^2}  r^{- \frac{\alpha + 1}{3} + \epsilon} \psi_r(\xi) | \hat \mu_1(\xi)|^2 d \xi dr \lesssim  \int_{\RR^2} |\xi|^{- \frac{\alpha + 1}{3} + \epsilon} | \hat \mu_1 (\xi)|^2 d \xi \sim I_{\beta} (\mu_1) $$

\noindent with $\beta = 2 - \frac{\alpha + 1}{3} + \epsilon$.   We know that $I_\beta(\mu_1) < \infty$ as long as $\beta < \alpha$, so we get the desired bound as long as

$$ 2 - \frac{\alpha + 1}{3} < \alpha. $$

\noindent This is equivalent to $\alpha > 5/4$. 
	
\end{proof}

\vskip.125in 

\begin{proof}[Proof of Proposition \ref{mug*sr}] Recall that

$$ \mug = M_0 \mu_1 + \sum_{j \ge 1, \tau} \sum_{T \in \TT_{j, \tau}, T \textrm{ good}} M_T \mu_1. $$

When we convolve with $\hat \sigma_r$ the only terms that remain are those with Fourier support intersecting the circle of radius $r$.  The interesting case is when $r > 10 R_0$.  We will return at the end to the case $R < 10 R_0$.  Assuming $r > 10 R_0$, $\mug * \hat \sigma_r$ is essentially equal to 

$$ \sum_{R_j \sim r} \sum_\tau \sum_{T\in \TT_{j, \tau} T \textrm{ good}} M_T \mu_1 * \hat \sigma_r. $$

Let $\eta_1$ be a bump function adapted to the unit ball.  We define 

$$ f_T = \eta_1 \left( M_T \mu_1 * \hat \sigma_r \right). $$

\noindent We claim that each $f_T$ is microlocalized in the way we would want to apply Corollary \ref{correfdec}.  If $T \in \TT_{j, \tau}$, then we let $\theta(T)$ be the 1-neighborhood of  $3 \tau \cap S^1_r$.  We claim that $\hat f_T$ is essentially supported in $\theta(T)$.  First we recall that $\widehat{M_T \mu_1}$ is essentially supported in $2 \tau$.  Therefore, the Fourier transform of $M_T \mu_1 * \hat \sigma_r$ is essentially supported in $2 \tau \cap S^1_r$.  Finally, the Fourier transform of $f_T$ is essentially supported in the 1-neighborhood of $2 \tau \cap S^1_r$, which is contained in $\theta$.  Note that $\theta$ is a rectangular block of dimensions roughly $r^{1/2} \times 1$.  

Next we claim that $f_T$ is essentially supported in $2T$.  We know that $M_T \mu_1$ is supported in $T$.  Let $\tilde \psi_\tau$ be a smooth bump function which is 1 on $2 \tau$ and rapidly decaying.  Since the Fourier transform of $M_T \mu_1$ is essentially supported on $2 \tau$, we have $M_T \mu_1 * \hat \sigma_r$ is essentially equal to $M_T \mu_1 * (\tilde \psi_\tau \sigma_r)^\wedge$.  It is standard to check by stationary phase that $(\tilde \psi_\tau \sigma_r)^\wedge$ is bounded by $\RD(r)$ on $B^2(1)$ outside of a tube of radius $r^{-1/2 + \delta}$ in the direction of $\tau$ passing through the origin.  So $M_T \mu_1 * (\tilde \psi_\tau \sigma_r)^\wedge$ is negligible on $B^2(1) \setminus 2T$.  So $f_T$ is essentially supported on $2T$.

We have $ \mug * \hat \sigma_r$ is essentially equal to $\sum_{T \textrm{ good}} f_T$.  Next we sort the $f_T$ according to their $L^p$ norms.

$$ \WW_\lambda:= \{ T: \| f_T \|_{L^p} \sim \lambda \}.$$

$$ f_\lambda := \sum_{T \in \WW_\lambda} f_T. $$

Since the number of scales $\lambda$ is $\lesssim \log r$, it suffices to prove the bound

$$ \int | f_\lambda (x) |^2 d\mu_2(x) \lesssim r^{- \frac{\alpha + 1}{3} + \epsilon} \| \hat \mu_1 \|_{L^2(d \sigma_r)}^2. $$

Next we divide the unit ball into $r^{-1/2}$-squares $q$ and sort them.  We let

$$ \QQQ_{\gamma, M} := \{ r^{-1/2} \textrm{ squares } q: \mu_2(q) \sim \gamma \textrm{ and } q \textrm{ intersects } \sim M \textrm{ tubes } T \in \WW_\lambda \}. $$

We let $Y_{\gamma, M} = \bigcup_{q \in \QQQ_{\gamma, M}} q$.  Since there are only $\sim \log^2 r$ choices of $\gamma, M$, it suffices to bound $\int_{Y_{\gamma,M}} | f_\lambda|^2 d\mu_2$.  Next we bound the measure of $Y_{\gamma, M}$.

\begin{lemma} \label{mu2Y} For any $\gamma, M$,
	
	$$ \mu_2 \left( Y_{\gamma, M} \right) \lesssim \frac{ | \WW_\lambda | r^{-1/2 + 100 \delta} }{M}. $$
\end{lemma}

\begin{proof} This is a double counting argument.  We count in two ways the size of the set of incidences,
	
	$$ I :=  \{ (q, T) \in \QQQ_{\gamma, M} \times \WW_\lambda: q \textrm{ intersects } T \}. $$
	
Since each tube $T \in \WW_\lambda$ is good, each tube $T$ has $\mu_2( 2 T) \lesssim r^{-1/2 + 100 \delta}$.  Therefore, the number of $q \in \QQQ_{\gamma, M}$ that $T$ intersects is $\lesssim 
 \gamma^{-1} r^{-1/2 + 100 \delta}$.  Therefore,
 
 $$ I \lesssim \gamma^{-1} r^{-1/2 + 100 \delta} | \WW_\lambda |. $$
 
 On the other hand, each cube $q \in \QQQ_{\gamma, M}$ intersects $\gtrsim M$ tubes $T \in \WW_\lambda$.  Therefore,
 
 $$ I \gtrsim | \QQQ_{\gamma, M} | M. $$
 
 Comparing these bounds for $I$, we get 
 
 $$|\QQQ_{\gamma, M}| \lesssim \frac{ r^{-1/2 + 100 \delta} | \WW_\lambda |}{\gamma M}. $$
 
 Since each cube $q \in \QQQ_{\gamma, M}$ has $\mu_2 (q) \sim \gamma$, we get
 
 $$ \mu_2 (Y_{\gamma, M}) \lesssim \frac{ | \WW_\lambda | r^{-1/2 + 100 \delta} }{M}. $$

\end{proof}

Now we are ready to bound

$$ \int_{Y_{\gamma,M}} | f_\lambda|^2 d\mu_2. $$

The Fourier support of $f_\lambda$ is essentially contained in the 1-neighborhood of $S^1_r$, and so $f_\lambda$ is (morally) locally constant at scale $\sim r^{-1}$.  Therefore we can replace $d \mu_2$ by $\mu_2 * \eta_{1/r}$, where $\eta_{1/r}$ is a bump function with integral 1 essentially supported on a ball of radius $1/r$.  Then we can use H\"older to bound

$$ \int_{Y_{\gamma,M}} | f_\lambda|^2 d\mu_2 \lesssim \left(\int_{Y_{\gamma,M}} | f_\lambda|^6 \right)^{1/3} \left(\int_{Y_{\gamma, M}} |\mu_2 * \eta_{1/r}|^{3/2} \right)^{2/3}. $$

To bound the first factor, we use Corollary \ref{correfdec} with $W = | \WW_\lambda|$ wave packets and multiplicity $M$.  We get

$$ \| f_\lambda \|_{L^6(Y_{\gamma, M})} \lessapprox \left(\frac{M}{| \WW_\lambda|} \right)^{1/3} \left( \sum_{T \in \WW_\lambda} \| f_T \|_{L^6}^2 \right)^{1/2}. $$

By Lemma \ref{mu2Y}, we can bound $M/ | \WW_\lambda|$ to get

$$ \lessapprox \left(\frac{r^{-1/2 + 100 \delta}}{\mu_2(Y_{\gamma, M})} \right)^{1/3} \left(\sum_{T \in \WW_\lambda} \| f_T \|_{L^6}^2 \right)^{1/2}. $$

To bound the second factor, we note that $\mu_2$ of a ball of radius $r^{-1}$ is at most $r^{-\alpha}$.  Therefore, 

$$\| \mu_2 * \eta_{1/r} \|_{L^\infty} \lesssim r^{2 - \alpha}. $$

And so

$$ \int_{Y_{\gamma, M}} | \mu_2 * \eta_{1/r} |^{3/2} \lesssim \left(r^{2 - \alpha} \right)^{1/2} \int_{Y_{\gamma, M}} d \mu_2 * \eta_{1/r} \sim r^{1 - \frac{\alpha}{2}} \mu_2(Y_{\gamma, M}). $$

Plugging in these two bounds, we get

$$ \int_{Y_{\gamma,M}} | f_\lambda|^2 d\mu_2 \lesssim r^{O(\delta)} r^{-1/3} \mu_2(Y_{\gamma, M})^{-2/3} \sum_{T \in \WW_\lambda} \| f_T \|_{L^6}^2 \cdot r^{\frac{2}{3} - \frac{\alpha}{3}} \mu_2(Y_{\gamma, M})^{2/3}. $$

Notice that the powers of $\mu_2(Y_{\gamma, M})$ cancel, leaving

$$ \int_{Y_{\gamma,M}} | f_\lambda|^2 d\mu_2 \lesssim r^{O(\delta)} r^{\frac{1 - \alpha}{3}} \sum_{T \in \WW_\lambda}  \| f_T \|_{L^6}^2.  $$

Next we record an elementary bound for $\| f_T \|_{L^6}$.  Since $f_T$ is essentially supported on $T$, $\| f_T \|_{L^6} \lesssim |T|^{1/6} \| f_T \|_{L^\infty} \sim r^{-1/12} \| f_T \|_{L^\infty}$.  Recall that

$$ f_T = \eta_1 (M_T \mu_1 * \hat \sigma_r) = \eta_1 \int_{S^1_r} \widehat{M_T\mu_1} d \sigma_r. $$

\noindent Since $\widehat {M_T \mu_1}$ restricted to $S^1_r$ is essentially supported on $\theta(T)$, we get

$$ \| f_T \|_{L^\infty} \lesssim \sigma_r (\theta(T))^{1/2} \| \widehat{M_T \mu_1} \|_{L^2(d \sigma_r)} \sim r^{-1/4}  \| \widehat{M_T \mu_1} \|_{L^2(d \sigma_r)}. $$

\noindent Therefore 

$$ \| f_T \|_{L^6} \lesssim r^{-1/3}  \| \widehat{M_T \mu_1} \|_{L^2(d \sigma_r)}. $$

Plugging into the last bound, we get

$$ \int_{Y_{\gamma,M}} | f_\lambda|^2 d\mu_2 \lesssim r^{O(\delta)} r^{\frac{-1 - \alpha}{3}} \sum_{T \in \WW_\lambda}  \| \widehat{M_T \mu_1} \|_{L^2(d \sigma_r)}^2.  $$

To finish the proof of Proposition \ref{mug*sr}, it just remains to check that 

\begin{equation} \label{orth} \sum_{R_j \sim r} \sum_{\tau} \sum_{T \in \TT_{j,\tau}}  \int |\widehat{M_T \mu_1}|^2 d \sigma_r \lesssim r^{-1}  \int | \hat \mu_1|^2 \psi_r d\xi. \end{equation}

Morally, we are showing that the $\widehat{M_T \mu_1}$ are approximately orthogonal with respect to $d \sigma_r$ and/or $\psi_r$.  The pieces $\widehat{M_T \mu_1}|_{S^1_r}$ correspond to the wave packet decomposition of $\hat \mu_1 * \hat \sigma_r$.  It's a standard fact that the wave packets in a wave packet decomposition are approximately orthogonal.  (For instance, see Section 3 of \cite{GuthII} for related orthogonality arguments.)  But because of the direction of \eqref{orth}, it takes some extra care to be completely rigorous.  In particular, it makes matters easier to put $\psi_r$ instead of $d \sigma_r$ on the right-hand side of (\ref{orth}), although we're not sure whether this is necessary.  Now we turn to the details.

Recall that $\psi_r$ is a weight function which is $\sim 1$ on the annulus $r-1 \le |\xi| \le r+1$ and then rapidly decaying.  Similarly, define $\psi_{j, \tau, r}$ to be a weight function which is roughly 1 on the intersection of $\tau$ with the annulus $r-1 \le |\xi| \le r+1$ and then rapidly decaying.  We recall that if $T \in \TT_{j, \tau}$, then $\widehat {M_T \mu_1}$ is rapidly decaying outside of $\tau$.   Since $M_T \mu_1$ is supported in $T \subset B^2(1)$, its Fourier transform is morally locally constant on scale 1.  Therefore, for any $T \in \TT_{j, \tau}$, we have

$$ \int |\widehat{M_T \mu_1}|^2 d \sigma_r \lesssim r^{-1} \int |\widehat{M_T \mu_1}|^2 \psi_{j, \tau, r} d \xi, $$

\noindent where the $r^{-1}$ comes because $\sigma_r$ is the normalized arc-length measure on $S^1_r$, which is equal to approximately $1/r$ times arc length measure.  Next we expand out

$$  r^{-1} \int |\widehat{M_T \mu_1}|^2 \psi_{j, \tau, r} d \xi = r^{-1}  \int | \hat \eta_T * (\psi_{j, \tau} \hat \mu_1) |^2 \psi_{j, \tau, r} d \xi. $$

Since $\hat \eta_T$ is essentially supported in a rectangle of dimensions $R^{1/2} \times 1$, with the long direction parallel to $S^1_r$ at points in $\tau \cap S^1_r$, we can bound

$$ r^{-1}  \int | \hat \eta_T * (\psi_{j, \tau} \hat \mu_1) |^2 \psi_{j, \tau, r} d \xi \lesssim r^{-1} \int | \hat \eta_T * (\tilde \psi_{j, \tau,r} \hat \mu_1) |^2 d \xi, $$

\noindent where $\tilde \psi_{j, \tau, r}$ is again rapidly decaying outside of $\tau \cap \{ r-1 \le |\xi| \le r\}$, but a bit more slowly than $\psi_{j, \tau, r}$.  The point of all these adjustments is that we can now apply Plancherel in a clean way:

$$ \sum_{T \in \TT_{j, \tau}} \int | \widehat{M_T \mu_1}|^2 d \sigma_r \lesssim r^{-1} \sum_{T \in \TT_{j, \tau}} \int | \eta_T|^2 | \tilde \psi_{j, \tau, r}^\vee * \mu_1|^2 dx. $$

Since any point $x$ lies in $\lesssim 1$ different $T \in \TT_{j,\tau}$, the last expression is bounded by

$$ \lesssim r^{-1} \int  | \tilde \psi_{j, \tau, r}^\vee * \mu_1|^2 dx = r^{-1} \int | \tilde \psi_{j, \tau, r} \hat \mu_1|^2 d \xi. $$

So now

\[
\begin{split}
\sum_{R_j \sim r} \sum_{\tau} \sum_{T \in \TT_{j, \tau}} \int|\widehat{M_T \mu_1}|^2 d \sigma_r 
\lesssim &r^{-1} \sum_{R_j \sim r, \tau} \int | \tilde \psi_{j, \tau, r} \hat \mu_1|^2 d \xi \\
=& r^{-1} \int \left(\sum_{R_j \sim r, \tau} \tilde \psi_{j, \tau,r}^2 \right)| \hat \mu_1|^2 d \xi. 
\end{split}
\]

The regions where $\tilde \psi_{j, \tau, r}$ are $\sim 1$ tile the annulus $r-1 \le |\xi | \le r+1$, with each point lying in $\lesssim 1$ regions.  Therefore, $\sum_{R_j \sim r,\tau}\tilde  \psi_{j, \tau, r}^2$ is $\sim 1$ on the annulus $r-1 \le \xi \le r+1$ and rapidly decaying elsewhere.  So $ \sum_{R_j \sim r,\tau}\tilde \psi_{j, \tau, r}^2 \lesssim \psi_r$, and we get 

$$ \sum_{T} \int |\widehat{M_T \mu_1}|^2 d \sigma_r \lesssim r^{-1} \int | \hat \mu_1|^2 \psi_r d \xi. $$

This gives (\ref{orth}) and finishes the proof for the main case $r > 10 R_0$.  

If $r < 10 R_0$, we give a more elementary estimate.  It is rather lossy, but the loss can be absorbed into the factor $C(R_0)$.  We write

$$ \int_{E_2} |\mug * \hat \sigma_r(x)|^2 d \mu_2(x) \le \| \mug * \hat \sigma_r \|_\infty^2 \le \| \widehat{\mug} \|_{L^1(d \sigma_r)}^2 \le \| \widehat{\mug} \|_{L^2(d \sigma_r)}^2. $$

Recall that $\mug$ is the sum of the good $M_T \mu_1$ while $\mu_1$ is the sum of all $M_T \mu_1$.  As we discussed above, the $M_T \mu_1$ are approximately orthogonal with respect to $\psi_r$, and so a similar argument to the one above shows that 

$$ \| \widehat{\mug} \|_{L^2(d \sigma_r)}^2 \lesssim r^{-1} \int |\hat \mu_1|^2 \psi_r d \xi. $$

\noindent Since $r \le 10 R_0$, we get

$$  \int_{E_2} |\mug * \hat \sigma_r(x)|^2 d \mu_2(x) \lesssim r^{-1} \int |\hat \mu_1|^2 \psi_r d \xi \lesssim  C(R_0) r^{- \frac{\alpha + 1}{3} + \epsilon} r^{-1} \int  | \hat \mu_1 |^2 \psi_r d \xi. $$

 \end{proof}

\section{Train track examples} \label{sec:traintracks}

As we mentioned in the introduction, when $\alpha < 4/3$, there are examples of measures where the Mattila integral is infinite, and the related $L^2$ integral in Liu's framework is also infinite.  The relevant sets look like several trains tracks.  These train track examples are based on the train track example in \cite{KT01} (page 151).  In this section, we discuss these measures and their properties.

\begin{proposition} For every $\alpha < 4/3$ and every $B$, there is a probability measure $\mu$ on $B^2(1)$ with the following properties:
	
\begin{enumerate}
	\item For any ball $B(x,r)$, $\mu(B(x,r)) \lesssim r^\alpha$.
	
	\item If $d(x,y) := |x-y|$, then
	
	$$ \| d_* (\mu \times \mu) \|_{L^2} > B. $$
		
\item If $d^x(y) := |x-y|$, then for every $x$ in the support of $\mu$, 

$$ \| d^x_*(\mu) \|_{L^2} > B. $$

\end{enumerate}	

\end{proposition}

\begin{proof}  Let $R$ be a large parameter.  Let $A_R$ be the set of points $(x_1,x_2) \in [0,1]^2$ where $0 \le x_1 \le R^{-1/2}$ and where, for some integer $M$, 
	
$$ M R^{-\alpha/2} \le x_2 \le  M R^{-\alpha/2}+ R^{-1} . $$

\noindent This set reminds me of a train track.  The slats of the train track are rectangles with dimensions $R^{-1/2} \times R^{-1}$, and there are $\sim R^{\alpha/2}$ slats evenly spaced inside a vertical rectangle of dimension $R^{-1/2} \times 1$.  We form a set $E_R$ by taking the union of $R^{\frac{\alpha -1}{2}}$ train tracks that are evenly spread.  To be definite, let us define $A_{R,l}$ to be the translate of $A_R$ by the vector $ (R^{- \frac{\alpha-1}{2}} l, 0)$. and then define $E_R$ to be the union of $A_{R,l}$ as $l$ goes from $0$ to $R^{\frac{\alpha-1}{2}}$.    (There is considerable freedom in how to take the union of the train tracks, and we could make similar examples with non-parallel train tracks also.)  Let $\mu_R$ be the normalized area measure on $E_R$.

First we check that $\mu_R (B(x,r)) \lesssim r^\alpha$.  The number of $R^{-1}$-boxes in $E_R$ is $R^{\frac{\alpha-1}{2}} R^{\alpha/2} R^{1/2} = R^\alpha$.  So we have to check that if $r = A R^{-1}$, then the number of $R^{-1}$ boxes in $E_R \cap B(x,r)$ is $\lesssim A^\alpha$.  If $A \le R^{1/2}$, then $B(x,r) \cap E_R$ is contained inside one train track.  The spacing between horizontal slats is $R^{-\alpha/2}$, and so the number of horizontal slats that intersect the ball $B(x,r)$ is at most

$$ \frac{r}{R^{-\alpha/2}} = A R^{-1} R^{\alpha/2}. $$
	
\noindent Each horizontal slat intersects $B(x,r)$ in at most $A$ $R^{-1}$-boxes.  So the total number of $R^{-1}$-boxes in $B(x,r)$ is at most

$$ A^2 R^{\frac{-2 + \alpha}{2}} = A^\alpha A^{2 - \alpha} R^{- \frac{2 - \alpha}{2}} \le A^\alpha, $$

\noindent where in the last inequality we used $A \le R^{1/2}$.  

Suppose $A \ge R^{1/2}$.  Morally, since the train tracks are spaced evenly, the estimates will be even better than for the case $A = R^{1/2}$.  Here are the details.  Since the spacing between train tracks is $R^{-\frac{\alpha-1}{2}}$, the number of train tracks that $B(x,r)$ intersects is at most

$$ r R^{\frac{\alpha-1}{2}} + 1= A R^{\frac{\alpha -3}{2}} + 1. $$

\noindent Within each train track, the number of slats that $B(x,r)$ intersects is at most

$$ r R^{\alpha/2} = A R^{\frac{\alpha-2}{2}}. $$

\noindent Each slat contains $R^{1/2}$ $R^{-1}$-boxes.  So the total number of $R^{-1}$-boxes in $B(x,r)$ is at most

$$ A^2 R^{\alpha - 2} + A R^{\frac{\alpha - 1}{2}}= A^\alpha A^{2 - \alpha} R^{\alpha - 2} +  A^\alpha A^{1- \alpha} R^{\frac{\alpha-1}{2}}\le A^\alpha, $$

\noindent where in the last inequality we used $R^{1/2} \le A \le R$.  

Next we estimate $\int |d_* (\mu \times \mu) |^2$.  The key point is that $d_*(\mu \times \mu)$ assigns a large measure to each interval $I_M = [M R^{- \alpha/2} - 2 R^{-1}, M R^{-\alpha/2} + 2 R^{-1}]$, where $M$ is an integer with $M \sim R^{\alpha/2}$.  Indeed, if $x$ is any point in $E_R$, and if $y$ lies in the same train track as $x$, in a horizontal slat which is $M$ steps from the horizontal slat containing $x$, then $|x-y| \in [M R^{- \alpha/2} - 2 R^{-1}, M R^{-\alpha/2} + 2 R^{-1}]$.  The $\mu_R$ measure of a single slat is $R^{-\alpha + 1/2}$, because the slat contains $R^{1/2}$ $R^{-1}$-boxes, which each have $\mu_R$ measure $R^{-\alpha}$.  Therefore,

$$ d_*(\mu \times \mu) (I_M) \gtrsim R^{-\alpha + \frac{1}{2}}. $$

\noindent By Cauchy-Schwarz,

$$ \int_{I_M} d_*(\mu \times \mu)^2 dt \gtrsim \frac{R^{-2 \alpha +1}}{|I_M|} \sim R^{- 2 \alpha + 2}. $$

\noindent The number of different $I_M$ is $\sim R^{\alpha/2}$, and so

$$ \int d_*(\mu \times \mu)^2 dt \gtrsim R^{- \frac{3}{2} \alpha + 2}. $$

\noindent If $\alpha < 4/3$, then the power of $R$ is positive, and $\int d_*(\mu \times \mu)^2 dt$ goes to infinity with $R$.

Finally we estimate $\int |d^x_*(\mu)|^2$.  The computation is similar: $d^x_*(\mu)$ assigns a large measure to each interval $I_M$ defined above.  In fact, just as above, 

$$d^x_*(\mu) (I_M) \gtrsim R^{-\alpha + \frac{1}{2}}. $$

\noindent because if $y$ lies in the slat of $E_R$ lying in the same train track as $x$ and $M$ horizontal slats from the horizontal slat containing $x$, then $d^x(y) = |x-y| \in I_M$, and the $\mu_R$ measure of this slat is $R^{-\alpha + 1/2}$.  Then just as above we get

$$ \int_{I_M} d^x_*(\mu)^2 dt \gtrsim \frac{R^{-2 \alpha +1}}{|I_M|} \sim R^{- 2 \alpha + 2} \textrm{ and }$$

$$ \int d^x_*(\mu)^2 dt \gtrsim R^{- \frac{3}{2} \alpha + 2}. $$

\noindent If $\alpha < 4/3$, then the right-hand side tends to infinity as desired.  \end{proof}

We can also take limits of these examples with different scalings.  Suppose that $R_j$ is a sequence of scales that goes to infinity rapidly.  Define

$$ E_j = \bigcap_{i = 1}^j E_{R_i}, \textrm{ and }$$

$$ E = E_\infty = \bigcap_{i=1}^\infty E_{R_i}. $$

Define $\mu_i$ to be $\mu_{R_i}$ restricted to $E_i$ and renormalized, and let $\mu = \mu_\infty$ be a weak limit of the measures $\mu_i$.  It is not hard to check that the Hausdorff dimension of $E$ is $\alpha$, that $\mu(B(x,r)) \lesssim r^\alpha$, and that $\| d_*(\mu \times \mu) \|_{L^2} = + \infty$ and $\| d^x_*(\mu) \|_{L^2} = + \infty$ for each $x \in E$.

\section{Generalization to other norms: proof of Theorem \ref{main2}} \label{sec:genmetric}

In this section, we consider the generalization of the Falconer problem where the Euclidean norm is replaced by other norms.  We will show that our main theorem generalizes to other norms as long as the unit ball of the norm is strictly convex and smooth.  

\begin{theorem} \label{mainK} Let $K$ be a symmetric convex body in $\mathbb{R}^2$ whose boundary $\partial K$ is $C^\infty$ smooth and has everywhere positive curvature bounded from above and below.  Let $\|\cdot \|_K$ denote the norm with unit ball $K$.  Define the pinned distance set
	
	$$ \Delta_{x,K}(E) := \{ \| x-y \|_K \}_{y \in E}. $$
	
If $E \subset \RR^2$ has Hausdorff dimension $> 5/4$, then there exists $x \in E$ so that the pinned distance set $\Delta_{x,K}(E)$ has positive Lebesgue measure.
\end{theorem}

Many parts of the proof work in the same way, and we will only discuss the required changes.  The most interesting new ingredient is a generalization of Liu's identity - Theorem \ref{Liu}.

Let us start by discussing the analogue of the train track examples for a general norm.  This will help us motivate the right way to decompose $\mu_1$ into pieces $M_T \mu_1$.  A train track consists of many parallel slats of dimensions $\sim R^{-1/2} \times R^{-1}$ contained in a larger rectangle of dimensions $\sim R^{-1/2} \times 1$.  In the original Euclidean case, the direction of the slats is perpendicular to the direction of the larger rectangle.  But to build an interesting example for the norm $\| \cdot \|_K$, the angle of the slats should be dictated by the geometry of $K$ in the following way.  Suppose that the long axis of the large rectange is parallel to a vector $v$.  By rescaling, we can assume that $v \in \partial K$.  Then build a train track where the slats are rectangles with long axis parallel to the tangent vector of $K$ at $v$.  In this case, if $x$ lies in one slat in the train track, and $y_1, y_2$ lie in the same slat at the opposite side of the train track, then $ \| x - y_1 \|_K = \| x - y_2 \|_K + O(R^{-1})$.  From here on, train track examples have the same properties as in the Euclidean case.

The angles of the slats have a nice interpretation in terms of the dual norm $K^*$.  We recall here some standard facts about dual norms.  Let $\| \cdot \|_{K^*}$ be the dual norm to $\| \cdot \|_K$, and let $K^*$ be the unit ball of the dual norm, which will also be smooth and strictly convex.  Recall that the dual norm is defined by

$$ \| \omega \|_{K^*} = \sup_{v \in K} \omega \cdot v. $$

\noindent By strict convexity, there is a unique $v \in K$ which achieves the supremum, which we denote by $v(\omega)$.  The plane $\omega \cdot v = \omega \cdot v(\omega)$ is tangent to $K$ at $v(\omega)$, and so we see that $\omega$ is normal to $\partial K$ at $v(\omega)$.  Similarly, for each vector $v$, there is a unique $\omega(v) \in \partial K^*$ so that $\omega(v) \cdot v = \| v \|_K$.  The plane $\{ \omega:\omega \cdot v = \omega(v) \cdot v \}$ is tangent to $K^*$ at $\omega(v)$ and so $v$ is normal to $\partial K^*$ at $\omega(v)$.  If $\omega \in \partial K^*$ and $v \in \partial K$, then $v \cdot \omega = 1$ if and only if $v = v(\omega)$ if and only if $\omega = \omega(v)$.  Therefore $\omega(v(\omega)) = \omega$ and $v(\omega(v)) = v$.  Since $\omega(v)$ is normal to $\partial K$ at $v$, the map $\omega: \partial K \rightarrow \partial K^*$ is essentially the Gauss map.  Because the curvature of $\partial K$ is $\sim 1$, the map $\omega$ is bilipschitz: for $v_1, v_2 \in \partial K, | \omega(v_1) - \omega(v_2)| \sim |v_1 - v_2|$.  Therefore the map $v: \partial K^* = \partial K$ is bilipschitz.  This shows that the curvature of $K^*$ is $\sim 1$.  

We can now generalize the decomposition $\mu_1 = \sum_T M_T \mu_1$ to the case of general norms $\| \cdot \|_K$, where $M_T \mu_1$ is designed to isolate the train track configurations described above.  We let $R_j$ and $\tau$ and $\psi_{j, \tau}$ be the same as in the Euclidean case.  But we redefine the tubes $\TT_{j, \tau}$.  For each $\tau$, consider $\partial (RK^*)$ where $R$ is chosen so that $\partial (RK^*) \cap \tau$ is non-empty (so $R \sim R_j$).  Then we let $\TT_{j, \tau}$ be a set of tubes $T$ with long direction perpendicular to $\partial(RK^*) \cap \tau$.  In other words, if $\omega \in \tau$, then the direction of the tubes $T$ is $v(\omega)$.  As before, the dimensions of the tubes are $R^{-1/2 + \delta} \times 1$ and the set of $T \in \TT_{j,\tau}$ covers $B^2(2)$.  We choose $\eta_T$ so that $\sum_{T \in \TT_{j, \tau}} \eta_T$ is 1 on $B^2(2)$.  Then we define, for each $T \in \TT_{j, \tau}$, 

$$ M_T f := \eta_T (\psi_{j, \tau} \hat f)^\vee. $$

We define good and bad tubes in the same way as in the Euclidean case, and as before we let

$$ \mug = M_0 \mu_1 + \sum_{j \ge 1} \sum_{\tau} \sum_{T \in \TT_{j,\tau}, T \textrm{ good}} M_T \mu_1. $$

In the Euclidean case, we studied the pushforwards $d^x_* \mu_1$ and $d^x_* \mug$ for $x \in E_2$.  In the case of general norms, we will use a small variation of the pushforward measure.  We need the small variation because of the way that the generalization of Liu's identity is stated, cf. Lemma \ref{LiuK} below.

Recall that $\sigma_t$ denotes the normalized arc length measure on the (Euclidean) circle of radius $t$.  Then as we saw above

$$ d^x_* (f darea) (t) = t \sigma_t * f(x). $$

Define $\sigma^K_t$ to be the normalized (Euclidean) arc length measure on $S^1_K(t)$ - the circle of radius $t$ in the norm $\| \cdot \|_K$.  Then we define

$$ T_{K, x} (f darea) (t) = t^{1/2} \sigma_t^K * f(x). $$

\noindent We note that the support of $T_{K,x} (f darea)$ is contained in $\Delta_{x,K}( \supp(f))$.  In particular, if $x \in E_2$, then the support of $T_{K,x} \mu_1$ is contained in $\Delta_{x,K}(E_1) \subset \Delta_{x, K}(E)$.   For comparison, if we let $d_K^x(y) = \| x - y \|_K$, we would have

$$ d^x_{K, *} (f darea) (t) = t (k \sigma_t^K) * f(x), $$

\noindent where $k(y)$ is a smooth positive function which only depends on the direction of $y$ -- i.e. $k(\lambda y) = k(y)$ for $\lambda \not= 0$.  If $x \in E_2$ and $y \in E_1$, then $\| x - y \|_K \sim 1$.  Therefore, both $d^x_{K,*} \mu_1$ and $T_{K,x} \mu_1$ are supported in $t \sim 1$, and by comparing the two formulas, we see that $T_{K,x} \mu_1 (t) \sim d^x_{K, *} \mu_1 (t)$.  In particular, this implies that $\int T_{K,x} \mu_1(t) dt \sim 1$.  

To prove Theorem \ref{mainK}, we have to prove analogues of Proposition \ref{mainest1} and Proposition \ref{mainest2}:

\begin{proposition} \label{mainest1K}  Let $K$ be a symmetric convex body in $\mathbb{R}^2$ whose boundary $\partial K$ is $C^\infty$ smooth and has everywhere positive curvature bounded from above and below.  If $\alpha > 1$, then there is a subset $E_2' \subset E_2$ so that $\mu_2(E_2') \ge 1 - \frac{1}{1000}$ and for each $x \in E_2'$,
	
	$$	\| T_{K,x} \mu_1 - T_{K,x} \mug \|_{L^1} < \frac{1}{1000}. $$
\end{proposition}

\begin{proposition} \label{mainest2K}  Let $K$ be a symmetric convex body in $\mathbb{R}^2$ whose boundary $\partial K$ is $C^\infty$ smooth and has everywhere positive curvature bounded from above and below.  If $\alpha > 5/4$, then 
	
	$$	\int_{E_2}  \| T_{K,x} \mug \|_{L^2}^2 d \mu_2(x) < + \infty. $$
\end{proposition}

\subsection{Proposition \ref{mainest1} for general norms}

In this section, we discuss the proof of Proposition \ref{mainest1K} the analogue of Proposition \ref{mainest1}.  We explain what needs to be modified in the proof of Proposition \ref{mainest1}.  The most significant part is the proof of the first lemma,  Lemma \ref{xnotinT}.   In the context of general norms, the lemma still holds with the same statement, but when we look at the proof we will need to use the way that $\tau$ and the direction of $T \in \TT_{j, \tau}$ are related to each other.

\begin{lemma} \label{xnotinTK} If $T \in \TT_{j, \tau}$, and $x \in E_2$, and $x \notin 2T$, then
	
	$$ \| T_{K,x} (M_T \mu_1) \|_{L^1} \lesssim \RD(R_j). $$
\end{lemma}

\begin{proof} 
	We will prove the stronger estimate that for every $t$:
	
	$$  T_{K,x} (M_T \mu_1) (t)  \lesssim \RD(R_j). $$
	
	Recall that
	
	\begin{equation} \label{dxMTK} T_{K,x}  M_T \mu_1(t) = t^{1/2}  \int_{S^1_K(x,t)} M_T \mu_1(y) d\sigma_t(y), \end{equation}
	
	\noindent where $S^1_K(x,t)$ is the circle around $x$ of radius $t$ in the norm $\| \cdot \|_K$ and $\sigma_t$ is the normalized arc length measure on it.
	
	We also recall that 
	
	$$ M_T \mu_1 = \eta_T (\psi_{j, \tau} \hat \mu_1)^\vee = \eta_T (\psi_{j, \tau}^\vee * \mu_1). $$
	
	Now $\psi_{j, \tau}^\vee$ is concentrated on a $R_j^{-1/2} \times R_j^{-1}$ rectangle centered at 0 and it decays rapidly outside that rectangle.  Since $x \in E_2$, the distance from $x$ to the support of $\mu_1$ is $\gtrsim 1$.  Therefore, $T_{K,x} (M_T \mu_1) (t)$ is tiny unless $t \sim 1$.
	
	To study the case when $t \sim 1$, we expand out $M_T \mu_1$:
	
	$$ M_T \mu_1 (y) = \eta_T(y) ( \psi_{j, \tau} \hat \mu_1 )^{\vee}(y)  = \eta_T(y)  \int e^{2 \pi i \omega y} \psi_{j, \tau}(\omega) \hat \mu_1( \omega) d \omega. $$
	
	Since $| \hat \mu_1(\omega) | \le 1$, and $\psi_{j, \tau}(\omega)$ is supported on $\tau$ and bounded by 1, it suffices to check that for each $\omega \in \tau$, 
	
	\begin{equation} \label{statphaseK} \int_{S^1_K(x,t)} \eta_T(y) e^{2 \pi i \omega y} d\sigma_t(y) \le \RD(R_j). \end{equation}
	
	We will prove this rapid decay by stationary phase.  After a coordinate rotation, we can assume that $\omega$ has the form $(0, \omega_2)$ with $\omega_2 \sim R_j$.  Let $T_0 \in \TT_{j, \tau}$ be the tube that passes through $x$.  The tubes of $\TT_{j, \tau}$ have long axis perpendicular to $\partial (RK^*)$ at a point in $\tau$.  In other words, the long axis of a tube in $\TT_{j, \tau}$ is parallel to $v(\omega)$ for $\omega \in \tau$ (up to angle $R^{-1/2}$).  The tube $T_0$ intersects $S^1_K(x,t)$ in two arcs, and on these arcs, the normal vector to $S^1_K(x,t)$ points in the direction $\omega(v(\omega)) = \omega$, which is vertical.  Now $T$ is not $T_0$ -- we know that $x \notin 2T$, and so the distance from $T$ to $T_0$ is $\gtrsim R^{-1/2 + \delta}$.  By the strict convexity of $K$, if $y \in S^1_K(x,t) \cap T$, then the normal vector to $S^1_K(x,t)$ at $y$ makes an angle $\gtrsim R^{-1/2 + \delta}$ with the vertical.
	
	The tube $T$ intersects $S^1_K(x,t)$ in one or two arcs.  We parametrize each arc as a graph -- either $y_2 = h(y_1)$ or $y_1 = h(y_2)$ -- over an interval $I(T)$.  By choosing one of these two options, we can assume that $h$ and all its derivatives are $\lesssim 1$ on $I(T)$.  Let us assume first that $y_2 = h(y_1)$ since this is the more interesting case.  Our integral becomes
	
	$$ \int_{I(t)} \eta_T(y_1, h(y_1)) e^{2 \pi i \omega_2 h(y_1)} J(y_1) dy_1. $$
	
	\noindent This is the same integral that appears in the proof of Lemma \ref{xnotinT}.   If $y_1 \in I(T)$, then $(y_1, y_2) \in S^1_K(x,t) \cap T$, and so the normal vector to $S^1_K(x,t)$ at $y$ makes an angle $\gtrsim R_j^{-1/2 + \delta}$ with the vertical.  Therefore, for $y_1 \in I(T)$, $| h'(y_1)| \gtrsim R_j^{-1/2 + \delta}$, just like in the proof of Lemma \ref{xnotinT}.  We can prove the desired estimate by the same stationary phase argument as in the proof of Lemma \ref{xnotinT}.
	
 If $y_1 = h(y_2)$, then we have a similar but easier integral:
	
	$$ \int_{I(T)} \eta_T(h(y_2), y_2) e^{2 \pi i \omega_2 y_2} J(y_2) dy_2. $$
	
	\noindent This integral is the same as the one appearing at the end of the proof of Lemma \ref{xnotinT}. 
	
	\end{proof}
	
It is also straightforward to check that if $f$ is supported on the annulus $\{ y: \| x- y\|_K \sim 1\}$, then

$$ \| T_{K,x} f \|_{L^1} \lesssim \| f \|_{L^1}. $$

The rest of the proof of Proposition \ref{mainest1} is unchanged.

\subsection{A general curve version of Liu's identity}

In the proof of Proposition \ref{mainest2}, the only ingredient that needs to be adjusted for general norms $K$ is Liu's $L^2$ identity -- Theorem \ref{Liu}.  The argument in \cite{Liu18} seemingly relies heavily on the rotation invariance of the circle.  We give a different approach to Theorem \ref{Liu}, and we show that it extends (modulo a negligible tail term) to more general metrics.  

The analogue of Liu's theorem is the following.

\begin{lemma} \label{LiuK} There is a smooth (not necessarily positive) measure $\sigma^{K^*}$ on $\partial K^* = S^1_{K^*}(1)$ so that the following holds.  Define a measure $\sigma^{K^*}_r$ on $S^1_{K^*}(r)$ by setting
	
	$$ \sigma^{K^*}_r (A)= \sigma^{K^*} (A/r). $$
	
	Suppose that $f(y)$ is supported in the annulus $\| x - y \|_{K} \sim 1$.  Then
	
\begin{equation} \label{eqLiuK} \int | f * \sigma^K_t (x)|^2 t dt \lesssim \int | f* \widehat{\sigma^{K^*}_r} (x)|^2 r dr + O(\|f\|_{\dot H^{-\frac{1}{2}+\epsilon}}^2 + \|f \|_{\dot H^{-\frac12}}^2 ). \end{equation}
	
\end{lemma}

\begin{proof} We abbreviate
	
	$$ F(t) = T_{K,x} f(t) = t^{1/2} \sigma^K_t * f(x). $$
	
Note that the left-hand side of (\ref{eqLiuK}) is $\int F(t)^2 dt$.   Also, because of the support condition on $f$, $F(t)$ is supported on $t \sim 1$.

We begin with an estimate for $\widehat{\sigma^K}$ which was derived by Herz in \cite{Herz}: When $|\xi| \ge 1$, we have

$$\widehat{\sigma^K}(\xi) = |\xi|^{-\frac{1}{2}} \kappa(\xi)^{-\frac{1}{2}} e^{2\pi i (\| \xi \|_{K^*} -\frac{1}{8})} + |\xi|^{-\frac{1}{2}} \kappa(-\xi)^{-\frac{1}{2}} e^{ -2\pi i (\| \xi \|_{K^*} -\frac{1}{8})} + O(|\xi|^{-\frac{3}{2}}),$$

\noindent where $\kappa(\xi)$ is the Gaussian curvature of $\partial K$ at $v(\xi)$ -- the vector with $\xi \cdot v(\xi) = \max_{v \in K} \xi \cdot v$.  Note that $\partial K$ is symmetric and so $\kappa(\xi) = \kappa(-\xi)$.  Also $\widehat{\sigma^K_t}( \xi) = \widehat{\sigma^K}(t \xi)$.  Therefore,

$$\widehat{\sigma^K_t}(\xi) = |t \xi|^{-1/2}\kappa(\xi)^{-\frac{1}{2}}( e^{2\pi i (t \| \xi \|_{K^*} -\frac{1}{8})}      +  e^{-2\pi i (t \| \xi \|_{K^*} -\frac{1}{8})}   )
+ O(|t\xi|^{-\frac{3}{2}}). $$

This bound holds for $|t \xi| \gtrsim 1$.  If $|t \xi| \lesssim 1$, then we have the simpler bound $|\widehat{ \sigma^K_t}(\xi)| \lesssim 1$ which gives the same expression with a remainder term of the form $|t \xi|^{-1/2}$ on the right-hand side.

Now we return to $F(t)$.  We have

$$ F(t) = t^{1/2} \sigma^K_t * f(x) = t^{1/2} \int e^{2 \pi i x \cdot \xi}  \widehat{\sigma^K_t} (\xi) \hat f (\xi) d \xi. $$

Plugging in the formula above for $\widehat{\sigma^K_t}$, we get two main terms and a remainder term -- for every $t \sim 1$, 

$$ F(t) = F_1(t) + F_2(t) + O\left(\int_{|\xi| \ge 1} | \widehat{f}(\xi) |  |\xi|^{-\frac{3}{2}} d\xi + \int_{|\xi| \le 1} | \hat f(\xi)| |\xi|^{-1/2} d \xi \right), $$

\noindent where

$$ F_1 (t) = e^{-i \frac{\pi}{4}} \int e^{2\pi i x \cdot\xi } \widehat{f}(\xi) \kappa(\xi)^{-\frac{1}{2}} |\xi|^{-\frac{1}{2}}  e^{2\pi i  t \| \xi \|_{K^*}} d\xi .$$
$$ F_2(t) = e^{i \frac{\pi}{4}} \int e^{2\pi i x \cdot\xi } \widehat{f}(\xi) \kappa(\xi)^{-\frac{1}{2}} |\xi|^{-\frac{1}{2}} e^{-2\pi i  t \| \xi \|_{K^*}} d\xi .$$

\noindent Notice that the $t^{1/2}$ in the expression $F(t) = t^{1/2} \sigma^K_t * f(x)$ cancelled the $t^{-1/2}$ in front of $\widehat{\sigma^K_t}(\xi)$.  This cancellation is the motivation for the expression $t^{1/2} \sigma^K_t * f$.  It leads to a simple formula for $\hat F_1$ and $\hat F_2$, which we can use to estimate $\int |F_1(t)|^2$ and $\int |F_2(t)|^2$.  

Before turning to Plancherel, let us mention that the formula for $F_i(t)$ makes sense for all real $t$.  

To find the formula for $\hat F_1(r)$, we will massage the definition of $F_1(t)$ into the form $F_1(t) = \int_0^\infty e^{2 \pi i r t} G(r) dr$.  Then it will follow that $\hat F_1(r) = G(r)$ (and that $\hat F_1$ is supported in $[0, \infty)$.)  Now we process the formula for $F_1$.  First, we write $\xi = r \theta$ where $r = \| \xi \|_{K^*}$ and $\theta \in S^1_{K^*} = \partial K^*$.  We can do a change of variables $d \xi = J(\theta) r dr d \theta, $ where $d \theta$ is arc length measure on $S^1_{K^*}$ and $J(\theta)$ is smooth and bounded.  Then we get

$$ F_1(t) =  e^{-i \frac{\pi}{4}} \int_0^\infty \int_{S^1_{K^*}} e^{2\pi i x \cdot\xi } \widehat{f}(\xi) \kappa(\theta)^{-\frac{1}{2}} |\xi|^{-\frac{1}{2}} J(\theta) d \theta e^{2\pi i  t r} r dr,$$

\noindent where $\xi = r \theta$.  We rewrote $\kappa(\xi) = \kappa(\theta)$ since $\kappa(\xi)$ only depends on the direction of $\xi$.  Up to another smooth factor $\tilde J(\theta)$, we have $|\xi| = r$, and so, after redefining $J(\theta)$, we have

$$ F_1(t) =  \int_0^\infty \left( e^{-i \frac{\pi}{4}}  \int_{S^1_{K^*}} e^{2\pi i x \cdot\xi } \widehat{f}(\xi) \kappa(\theta)^{-\frac{1}{2}}  J(\theta) d \theta r^{1/2}  \right) e^{2\pi i  t r} dr .$$

Therefore,

$$ \hat F_1(r) =  e^{-i \frac{\pi}{4}} \int_{S^1_{K^*}} e^{2\pi i x \cdot\xi } \widehat{f}(\xi) \kappa(\theta)^{-\frac{1}{2}}  J(\theta) d \theta r^{1/2}.$$

Now define $\sigma^{K^*} = \kappa(\theta)^{-1/2} J(\theta) d \theta$, a smooth measure on $S^1_{K^*}$, and we have

$$ \hat F_1(r) =   e^{-i \frac{\pi}{4}}  f* \widehat{\sigma^{K^*}_r} (x) r^{1/2}. $$

Now by Plancherel, we get

$$ \int | F_1(t)|^2 dt = \int | \hat F_1(r)|^2 dr = \int_{0}^\infty | f* \widehat{\sigma^{K^*}_r} (x)|^2 r dr. $$

A similar bound applies to $F_2$.  

Finally, we put it all together:

$$ \int |F(t)|^2 dt \sim \int_{t \sim 1} |F(t)|^2 dt \lesssim \int |F_1(t)|^2 dt + \int |F_2(t)|^2 dt + \textrm{ Remainder}, $$

\noindent where

$$ | \textrm{Remainder} | \lesssim \int_{|\xi| \ge 1} | \widehat{f}(\xi) |  |\xi|^{-\frac{3}{2}} d\xi + \int_{|\xi| \le 1} | \hat f(\xi)| |\xi|^{-\frac12} d \xi. $$

The main two terms are $\lesssim \int_0^\infty |f * \widehat{\sigma^{K^*}_r}(x)|^2 r dr$.  The remainder terms are controlled by Cauchy-Schwarz:

$$ \left( \int_{|\xi| \ge 1} | \widehat{f}(\xi) |  |\xi|^{-\frac{3}{2}} d\xi \right)^2 \lesssim \left( \int | \hat f(\xi)|^2 |\xi|^{-1 + 2 \epsilon} d\xi \right) \left(\int_{|\xi| \ge 1} |\xi|^{-2 - 2 \epsilon} d \xi \right) \lesssim \| f \|_{\dot H^{-\frac12 + \epsilon}}^2.$$

$$  \left( \int_{|\xi| \le 1} | \widehat{f}(\xi) |  |\xi|^{-\frac{1}{2}} d\xi \right)^2 \lesssim \int_{|\xi| \le 1} |\hat f|^2 |\xi|^{-1} d \xi \le \| f \|_{\dot H^{-\frac12}}^2. $$

\end{proof}

We wish to apply this Lemma with $f= \mug$.  Our $\mug$ is rapidly decaying outside of a tiny neighborhood of $E_1$, and so if $x \in E_2$, $\mug$ is essentially supported in an annulus of the form $\| x-y \|_K \sim 1$.  So we can apply Lemma \ref{LiuK}, and to make use of it we just need to check that the remainder terms are finite: in other words

$$ \| \mug \|_{\dot H^{-s}} < \infty, $$

\noindent for $s = 1/2$ or $1/2 - \epsilon$.  

Let us first check that $ \| \mug \|_{\dot H^{-s}} \lesssim \| \mu_1 \|_{\dot H^{-s}}$.  Indeed,

$$ \| \mug \|_{\dot H^{-s}}^2 \lesssim \sum_{j, \tau} R_j^{-2s} \sum_{T \in \TT_{j, \tau}, T \textrm{ good}}  \int | \widehat{M_T \mu_1} |^2.$$

Applying Plancherel, 

$$ \sum_{T \in \TT_{j, \tau}, T \textrm{ good}}  \int | \widehat{M_T \mu_1} |^2 = \sum_{T \in \TT_{j, \tau}, T \textrm{ good}}  \int | \eta_T |^2 | (\psi_{j,\tau} \hat \mu_1)^\vee |^2 \lesssim $$

$$ \lesssim  \int | (\psi_{j,\tau} \hat \mu_1)^\vee |^2  = \int | \psi_{j,\tau}|^2 |\hat \mu_1|^2. $$

Plugging into the above, we see that

$$ \| \mug \|_{\dot H^{-s}}^2 \lesssim \sum_{j, \tau} R_j^{-2s}  \int | \psi_{j,\tau}|^2 |\hat \mu_1|^2 \lesssim \int |\xi|^{-2s} |\hat \mu_1|^2 = \| \mu_1 \|_{\dot H^{-s}}^2. $$

The norm $\| \mu_1 \|_{\dot H^{-s}}$ is related to the dimension $\alpha$ as follows.  Recall that the $\beta$-dimensional energy of a measure $\mu$ is given by

$$ I_\beta (\mu) := \int |x-y|^{-\beta} \mu(x) \mu(y). $$

There is also a Fourier representation for $I_\beta(\mu)$ (cf. Proposition 8.5 of \cite{W03}): if $\mu$ is a measure on $\RR^n$, then 

$$ I_{\beta} (\mu) = c_{n, \beta}  \int_{\RR^n} |\xi|^{-(n - \beta)} |\hat \mu(\xi) |^2 d \xi. $$

In particular $\| \mu_1 \|_{\dot H^{-s}}^2 = \int |\xi|^{-2s} | \hat \mu_1 |^2 = I_{2-2s} (\mu_1)$.  If a measure $\mu$ on the unit ball obeys $\mu (B(x,r)) \lesssim r^\alpha$, then $I_\beta(\mu)$ is finite for every $\beta < \alpha$  (cf. Lemma 8.3 of \cite{W03}).  In particular, $I_\beta(\mu_1) < \infty$ for every $\beta < \alpha$.  Therefore, $\| \mu_1 \|_{\dot H^{-s}} < \infty$ whenever $2-2s < \alpha$ or $s > 1 - \alpha/2$.  In particular, if $\alpha > 1$, then the remainder terms are controlled and we get

$$\| T_{K,x} \mug \|_{L^2}^2 = \int | \mug * \sigma^K_t (x)|^2 t dt \lesssim \int_0^\infty | \mug * \widehat{ \sigma^{K^*}_r} (x)|^2 r dr + O(1). $$

Integrating with respect to $d \mu_2(x)$, we get

$$ \int_{E_2}\| T_{K,x} \mug \|_{L^2}^2 \lesssim \int_0^\infty \left(\int_{E_2} | \mug * \widehat{ \sigma^{K^*}_r} (x) |^2 d \mu_2(x) \right) r dr + O(1). $$

As in the proof in the Euclidean case, we bound the inner integral using Corollary \ref{correfdec}.  The proof is essentially the same as in the Euclidean case, but when we check that each piece $f_T$ is microlocalized correctly, we have to take into account the angles between the tubes $T \in \TT_{j, \tau}$ and the normal vector to $\partial K^*$ in the $\tau$ direction.  Here are the details.

$$ \mug = M_0 \mu_1 + \sum_{j \ge 1, \tau} \sum_{T \in \TT_{j, \tau}, T \textrm{ good}} M_T \mu_1. $$

When we convolve with $ \widehat{ \sigma^{K^*}_r} $, the only terms that remain are those with Fourier support intersecting $S^1_{K^*}(r)$.  So $\mug *  \widehat{ \sigma^{K^*}_r} $ is essentially equal to 

$$ \sum_{R_j \sim r} \sum_\tau \sum_{T\in \TT_{j, \tau} T \textrm{ good}} M_T \mu_1 *  \widehat{ \sigma^{K^*}_r} . $$

Let $\eta_1$ be a bump function adapted to the unit ball.  We define 

$$ f_T = \eta_1 \left( M_T \mu_1 *  \widehat{ \sigma^{K^*}_r}  \right). $$

\noindent We claim that each $f_T$ is microlocalized in the way we would want to apply Corollary \ref{correfdec}.  If $T \in \TT_{j, \tau}$, then we let $\theta(T)$ be the 1-neighborhood of  $3 \tau \cap S^1_{K^*}(r)$.  We claim that $\hat f_T$ is essentially supported in $\theta(T)$.  First we recall that $\widehat{M_T \mu_1}$ is essentially supported in $2 \tau$.  Therefore, the Fourier transform of $M_T \mu_1 *  \widehat{ \sigma^{K^*}_r} $ is essentially supported in $2 \tau \cap S^1_{K^*}(r)$.  Finally, the Fourier transform of $f_T$ is essentially supported in the 1-neighborhood of $2 \tau \cap S^1_{K^*}(r)$, which is contained in $\theta(T)$.  Note that $\theta(T)$ is a rectangular block of dimensions roughly $r^{1/2} \times 1$.  

Next we claim that $f_T$ is essentially supported in $2T$.  We know that $M_T \mu_1$ is supported in $T$.  Let $\tilde \psi_\tau$ be a smooth bump function which is 1 on $2 \tau$ and rapidly decaying.  Since the Fourier transform of $M_T \mu_1$ is essentially supported on $2 \tau$, we have $M_T \mu_1 *  \widehat{ \sigma^{K^*}_r} $ is essentially equal to $M_T \mu_1 * (\tilde \psi_\tau \sigma^{K^*}_r )^\wedge$.  It is standard to check by stationary phase that $(\tilde \psi_\tau \sigma^{K^*}_r )^\wedge$ is bounded by $\RD(r)$ on $B^2(1)$ outside of a tube of radius $r^{-1/2 + \delta}$ in the direction which is normal to $S^1_{K^*}(r)$ in $\tau$.  By construction, the tube $T$ also goes in this direction.  Therefore, $M_T \mu_1 * (\tilde \psi_\tau \sigma^{K^*}_r)^\wedge$ is negligible on $B^2(1) \setminus 2T$.  So $f_T$ is essentially supported on $2T$.

The rest of the proof of Proposition \ref{mainest2} is the same as in the Euclidean case.  When we apply Theorem \ref{refdec}, the surface $S$ that we use is $\partial K^*$.  Since Theorem \ref{refdec} only requires $S$ to be a $C^2$ hypersurface with all extrinsic curvatures $\sim 1$, it applies to $\partial K^*$.

\subsection{Norms with some points of vanishing curvature}

Theorem \ref{mainK} applies to norms $\| \cdot  \|_K$ where $\partial K$ has strictly positive curvature everywhere.  This assumption rules out the $l^p$ norms for all $p \not=2$.  If $1 < p < \infty$, and $p \not= 2$, then there are finitely many points on the boundary of the unit ball where the curvature vanishes.  Theorem \ref{mainK} can be generalized to the case when $\partial K$ is smooth and the curvature vanishes at finitely many points by a small extra trick.  We first set up $E_1$, $E_2$, $\mu_1$, and $\mu_2$ as usual, but then we refine them to avoid the directions where the curvature of $K$ vanishes.  Let $r_0$ be a small radius that we can choose later.  Let $B_1$ be any ball of radius $r_0$ with $\mu_1(B_1) > 0$, and replace $E_1$ by $E_1 \cap B_1$.  Then cover $E_2$ with balls of radius $r_0$.  We call a ball $B$ from this covering bad if there are points $x_1 \in B_1$ and $x_2 \in B$ so that the vector $x_2 - x_1$ is parallel to a vector $v \in \partial K$ where the curvature of $\partial K$ vanishes.  The number of bad balls is $\lesssim r_0^{-1}$.  Since $\mu_2(B(x,r)) \lesssim r^\alpha$ with $\alpha > 1$, we can find a good ball $B_2$ with $\mu_2(B_2) > 0$.  Now we replace $E_2$ by $E_2 \cap B_2$.  We redefine $\mu_1$ and $\mu_2$ to be supported on our new smaller sets $E_1$ and $E_2$.  

If $x_i \in B_i$, then the vector $(x_2 - x_1) / | x_2 - x_1|$ lies in an arc of $\partial K$ of length $\sim r_0$ which avoids all the flat points of $\partial K$.  Now we define $\tilde K$ to be a different symmetric convex body so that $\partial \tilde K$ includes this arc of $\partial K$ but $\partial \tilde K$ is smooth with strictly positive curvature everywhere.  We can apply our proof to $\| \|_{\tilde K}$.  It gives us a point $x \in E_2$ so that $d^x_{\tilde K} (E_1)$ has positive Lebesgue measure.  But if $x_1 \in E_1$ and $x \in E_2$, then $\| x - x_1 \|_K = \| x - x_1 \|_{\tilde K}$, and so $d^x_{K} (E_1)$ has positive Lebesgue measure also.

\vskip.125in 

\section{Applications of the main results to the Erd\H os distance problem for general norms}  \label{sec:erdos}

\vskip.125in 

The purpose of this section is to prove Corollary \ref{falconertoerdosthm} and extend it to a more general collection of point sets. The following definition is due to the second listed author, Rudnev and Uriarte-Tuero (\cite{IRU14}). 

\begin{definition} \label{sadaptable} Let $P$ be a set of $N$ points contained in ${[0,1]}^d$. Define the measure
\begin{equation} \label{pizdatayamera} d \mu^s_P(x)=N^{-1} \cdot N^{\frac{d}{s}} \cdot \sum_{p \in P} \chi_B(N^{\frac{1}{s}}(x-p)) dx, \end{equation} where $\chi_B$ is the indicator function of the ball of radius $1$ centered at the origin. We say that $P$ is \emph{$s$-adaptable} if there exists $C$ independent of $N$ such that 
\begin{equation} \label{sadaptenergy} I_s(\mu_P)=\int \int {|x-y|}^{-s} d\mu^s_P(x) d\mu^s_P(y) \leq C. \end{equation} 
\end{definition} 

It is not difficult to check that (\ref{sadaptenergy}) is equivalent to the condition 
\begin{equation} \label{discreteenergy} \frac{1}{N^2} \sum_{p \not=p'} {|p-p'|}^{-s} \leq C. \end{equation} 

\noindent It is also easy to check that if the distance between any two points of $P$ is $\gtrsim N^{-1/2}$, then (\ref{discreteenergy}) holds for any $s \in [0,d)$, and hence $P$ is $s$-adaptable.

We will prove that if $P$ is $s$-adaptable, then for some $x \in P$, $| \Delta_{K,x}(P)| \gtrapprox N^{4/5}$.  As a special case, this implies Corollary \ref{falconertoerdosthm}.

Fix $s>\frac{5}{4}$ and define $d\mu^s_P$ as above. Note that the support of $d\mu^s_P$ is 
$P^{N^{-\frac{1}{s}}}$, the $N^{-\frac{1}{s}}$-neighborhood of $P$. Since $I_s(\mu_P^s)$ is uniformly bounded, the proof of Theorem \ref{main2} implies that there exists $x_0 \in P^{N^{-1/s}}$ so that
$$ {\mathcal L}(\Delta_{K, x_0}(P^{N^{-\frac{1}{s}}})) \ge c>0,$$ 

\noindent where the constant $c$ only depends on the value of $C$ in (\ref{discreteenergy}).  

Let $x$ be a point of $P$ with $|x-x_0| \le N^{-1/s}$.  It follows that for any $y$, $\| x_0 - y \|_K = \| x-y\|_K + O(N^{-1/s})$.  Let $E_{N^{-1/s}} \left(\Delta_{K,x}(P) \right)$ be the smallest number of $N^{-1/s}$-intervals needed to cover $\Delta_{K,x}(P)$.  We know that $\Delta_{K, x_0} (P^{N^{-1/s}})$ is contained in the $O(N^{-1/s})$ neighborhood of $\Delta_{K,x}(P)$, and so 

$$ {\mathcal L}(\Delta_{K, x_0}(P^{N^{-\frac{1}{s}}})) \lesssim N^{-\frac{1}{s}} E_{N^{-1/s}} \left(\Delta_{K,x}(P) \right). $$

Then our lower bound on ${\mathcal L}(\Delta_{K, x_0}(P^{N^{-\frac{1}{s}}}))$ gives 

$$E_{N^{-1/s}} \left(\Delta_{K,x}(P) \right) \gtrsim N^{1/s}. $$

In other words, $\Delta_{K,x}(P)$ contains $\gtrsim N^{1/s}$ different distances that are pairwise separated by $\gtrsim N^{-1/s}$.  In particular, $|\Delta_{K,x}(P)| \gtrsim N^{1/s}$.  Since this holds for every $s > 5/4$, we get $| \Delta_{K,x}(P)| \gtrapprox N^{4/5}$ as desired.

\section{Appendix: discussion of the lower bound on the upper Minkowski dimension of $\Delta_{x,K}(E)$ in Remark \ref{ksanalogrmk}} 

\vskip.125in 

Let $\rho$ be a smooth cut-off function supported in the ball of radius $2$ and equal to $1$ in the ball of radius $1$ centered at the origin. Let $\rho_{\delta}(x)=\delta^{-d} \rho(\delta^{-1}x)$. Following the argument in (\ref{yes}) with $\mu_{1,good}$ replaced by $\mu_{1,good}*\rho_{\delta}$, we see that the Lebesgue measure of the $\delta$-neighborhood of $\Delta_{x,K}(E)$ is bounded from below by 
$$ \frac{\left(1-\frac{2}{1000}\right)^2}{\int {|d_{*}^{x}\mu_{1,good}*\rho_{\delta}|}^2}.$$

Following (\ref{nirvanaisnear}) with $\mu_{1,good}$ replaced by $\mu_{1,good}*\rho_{\delta}$, we see that the expression above is bounded from below by $C \delta^{\frac{5}{3}-\frac{4 \alpha}{3}+\epsilon}$, hence there exists $x \in E$ such that the upper Minkowski dimension of $\Delta_{x,K}(E)$ is bounded from below by 
$$ 1-\left(\frac{5}{3}-\frac{4 \alpha}{3} \right)=\frac{4}{3}\alpha-\frac{2}{3},$$ as claimed. 

\vskip.125in 

It would be interesting to obtain a lower bound on the Hausdorff dimension of $\Delta_{x,K}(E)$. If $\mu_{1,good}$ were positive, it would be sufficient to show that 
\begin{equation} \label{near1key} \int_{E_2} I_{\gamma} d^x_*(\mug) d\mu_2(x) \end{equation} is bounded with $\gamma<\frac{4}{3}\alpha-\frac{2}{3}$. This estimate follows from the same argument as in (\ref{nirvanaisnear}) above. Unfortunately, in view of the fact that $\mu_{1,good}$ is complex valued, the estimate (\ref{near1key}) does not appear to be sufficient to draw the desired conclusion.


\vskip.25in 

\begin{thebibliography}{3}

\bibitem{BBCRV07} J. Barcelo, J. Bennett, A. Carbery, A. Ruiz, C. Vilela, {\it Some special solutions of the Schr\"odinger euqation}, Indiana Univ. Math. J. \textbf{56} (2007), no. 4, 1581-1591.

\bibitem{B94} J. Bourgain, {\it Hausdorff dimension and distance sets}, Israel J. Math. 87 (1994), no. 1-3, 193-201.

\bibitem{B03} J. Bourgain, {\it On the Erd\H os-Volkmann and Katz-Tao ring conjectures}, Geom. Funct. Anal., \textbf{13} (2):334-365, (2003). 

\bibitem{BD15} J. Bourgain and C. Demeter, The proof of the $l^2$ decoupling conjecture. Ann. of Math. (2) 182 (2015), no. 1, 351-389.

\bibitem{DGL17} X. Du, L. Guth, and X. Li, {\it A sharp Schr\"odinger maximal estimate in $\mathbb{R}^2$}, Ann. of Math. \textbf{186} (2017), 607-640.

\bibitem{DGLZ18} X. Du, L. Guth, X. Li, and R. Zhang, {\it Pointwise convergence of Schr\"odinger solutions and
multilinear refined Strichartz estimate}, (arXiv:1803.01720) (2018).

\bibitem{DGOWWZ18} X. Du, L. Guth, Y. Ou, H. Wang, B. Wilson, and R. Zhang, {\it Weighted restriction estimates and application to Falconer distance set problem}, (arXiv:1802.10186) (2018). 

\bibitem{DZ18} X. Du and R. Zhang, {\it Sharp $L^2$ estimate of Schr\"odinger maximal function in higher dimensions}, (arXiv:1805.02775) (2018).

\bibitem{EIT11} S. Eswarathasan, A. Iosevich and K. Taylor, {\it Fourier integral operators, fractal sets and the regular value theorem}, Adv. Math. {\bf 228} (2011), 2385-2402.

\bibitem{Erd05} B. Erdo\u{g}an, {\it A bilinear Fourier extension theorem and applications to the distance set problem}, Int. Math. Res. Not. (2005), no. 23, 1411-1425.

\bibitem{Erd45} P. Erd\"os {\it On sets of distances of n points,} Amer. Math. Monthly. \textbf{53} (1946), 248--250.

\bibitem{Falc86} K. J. Falconer, {\it On the Hausdorff dimensions of distance sets}, Mathematika \textbf{32} (1985), no. 2, 206-212 (1986). 

\bibitem{Gar04} J. Garibaldi, {\it Erd\"os distance problem for convex metrics}, Thesis (Ph.D.)?University of California, Los Angeles. (2004). 

\bibitem{GuthII} L. Guth, {\it Restriction estimates using polynomial partitioning II}, (arXiv:1603.04250) (2016).

\bibitem{GK15} L. Guth, N. H. Katz, {\it On the Erd\H os distinct distance problem in the plane}, Ann. of Math. (2) {\bf 181} (2015), no. 1, 155-190.

\bibitem{Herz} C. S. Herz, {\it Fourier transforms related to convex sets}, Ann. of Math. (2) {\bf 75} (1962), no. 1, 81-92.

\bibitem{HI05} S. Hofmann and A. Iosevich {\it Circular averages and Falconer/Erd\"os distance conjecture in the plane for random metrics} Proc. Amer. Mat. Soc. \textbf{133} (2005) 133-144. 

\bibitem{IRU14} A. Iosevich, M. Rudnev and I. Uriarte-Tuero, {\it Theory of dimension for large discrete sets and applications}, Math. Model. Nat. Phenom. \textbf{9} (2014), no. 5, 148-169.

\bibitem{IL05} A. Iosevich, I. {\L}aba, {\sl K-distance sets, Falconer conjecture, and discrete analogs},
Integers: Electronic Journal of Combinatorial Number Theory, \textbf{5} (2005), \#A08 (hardcopy in: Topics in Combinatorial Number Theory: Proceedings of the Integers Conference 2003 in Honor of Tom Brown, DIMATIA, ITI Series, vol. 261).

\bibitem{ILiu17} A. Iosevich and B. Liu, {\it Pinned distance problem, slicing measures and local smoothing estimates}, (arXiv:1706.09851) (2017). 

\bibitem{IR07} A. Iosevich and M. Rudnev, {\it Distance measures for well-distributed sets}, Discrete Comput. Geom. \textbf{38}, (2007), 61-80.

\bibitem{IS16} A. Iosevich and S. Senger, {\it Sharpness of Falconer's $\frac{d+1}{2}$ estimate}, Ann. Acad. Sci. Fenn. Math. \textbf{41} (2016), no. 2, 713-720.

\bibitem{KT01} N. Katz and T. Tao, Some connections between Falconer's distance set conjecture and sets of Furstenburg type. New York J. Math. 7 (2001), 149-187.

\bibitem{KT04} N. Katz and G. Tardos, {\it A new entropy inequality for the Erd?s distance problem}, Towards a theory of geometric graphs, 119-126, Contemp. Math., \textbf{342}, Amer. Math. Soc., Providence, RI, (2004).

\bibitem{KS18} T. Keleti and P. Shmerkin, {\it New bounds on the dimensions of planar distance sets}, (arXiv:1801.08745) (2018). 

\bibitem{Liu18} B. Liu, {\it An $L^2$-identity and pinned distance problem}, (arXiv:1802.00350), (2018). 

\bibitem{Mat85} P. Mattila, {\it On the Hausdorff dimension and capacities of intersections}, Mathematika \textbf{32}, (1985), 213-217.

\bibitem{Mat87} P. Mattila, {\it Spherical averages of Fourier transforms of measures with finite energy: dimensions of intersections and distance sets}, Mathematika \textbf{34}, (1987), 207-228.

\bibitem{M52} L. Moser. {\it On the different distances determined by n points}, The American Mathematical Monthly 59.2 (1952): 85-91.

\bibitem{O17} T. Orponen, {\it On the distance sets of Ahlfors-David regular sets}, Adv. Math., \textbf{307}, 1029-1045, (2017). 

\bibitem{O17b} T. Orponen, {\it On the dimension and smoothness of radial projections}, Preprint, arXiv:1710.11053v2, 2017.

\bibitem{PS00} Y. Peres and W. Schlag, {\it Smoothness of projections, Bernoulli convolutions and the dimension of exceptions}, Duke Math J. \textbf{102}, 193-251, (2000). 

\bibitem{Shm17A} P. Shmerkin. On distance sets, box-counting and Ahlfors regular sets. Discrete Anal. 22 pp, (2017). 

\bibitem{Shm17B} P. Shmerkin, {\it On the Hausdorff dimension of pinned distance sets}, (arXiv:1706.00131), (2017).

\bibitem{SV08} J. Solymosi and V. Vu, {\it Near optimal bounds for the Erd\H os distinct distances problem in high dimensions} Combinatorica \textbf{28} (2008), no. 1, 113-125. 

\bibitem{St} E. Stein, {\it Harmonic Analysis}, Princeton University Press, 1993.

\bibitem{W99} T. Wolff, {\it Decay of circular means of Fourier transforms of measures}, Int. Math. Res. Not.   (1999), no. 10, 547--567.

\bibitem{W03} T. Wolff, {\it Lectures on Harmonic Analysis}, University Lecture Series, Volume 29, American Mathematican Society, 2003. 

\end{thebibliography}
\end{document}